\documentclass[12pt]{amsart}
\usepackage[english]{babel}
\usepackage[utf8]{inputenc}
\usepackage{lipsum}
\usepackage{mathrsfs}
\usepackage{stix}
\usepackage{fullpage}
\usepackage{mathtools}
\usepackage{amsmath}
\usepackage{tikz-cd}
\newtheorem{defi}{Definition}[section]

\newtheorem{lema}[defi]{Lemma}
\newtheorem{teo}[defi]{Theorem}
\newtheorem{rem}[defi]{Remark}

\newtheorem{coro}[defi]{Corollary}
\newtheorem{pro}[defi]{Proposition}
\newtheorem*{rem*}{Remark}

\newcommand{\hb}{\mathcal{H}}

\newcommand{\interior}[1]{%
  {\kern0pt#1}^{\mathrm{o}}%
}

\newcommand{\T}{\mathbb{T}}

\newcommand{\C}{\mathbb{C}}
\newcommand{\Q}{\mathbb{Q}}

\newcommand{\R}{\mathbb{R}}
\newcommand{\N}{\mathbb{N}}
\newcommand{\Z}{\mathbb{Z}}
\newcommand{\norm}[1]{\left\lVert#1\right\rVert}
\newcommand{\esp}{\text{  }}

\usepackage[colorlinks=false]{hyperref}
\renewcommand\eqref[1]{(\ref{#1})} 

\begin{document}

\title[SPECTRAL PROPERTIES OF PSEUDO-DIFFERENTIAL OPERATORS OVER COMPACT VILENKIN GROUPS]
 {SPECTRAL PROPERTIES OF PSEUDO-DIFFERENTIAL OPERATORS OVER THE COMPACT GROUP OF $p$-ADIC INTEGERS AND COMPACT VILENKIN GROUPS}

\author{  J.P. Velasquez-Rodriguez
}

\newcommand{\Addresses}{{
  \bigskip
  \footnotesize
  J.P.~VELASQUEZ-RODRIGUEZ, \textsc{Department of Mathematics: Analysis, Logic and Discrete Mathematics, Ghent University, Belgium.}\par\nopagebreak
  \textit{E-mail address:} \texttt{JuanPablo.VelasquezRodriguez@UGent.be}
}
}
\thanks{The 
	author was supported by the FWO Odysseus 1 grant G.0H94.18N: Analysis and Partial Differential Equations.}




\date{\today}
\begin{abstract}

In this paper. we study properties such as $L^r$-boundedness, compactness, belonging to Schatten classes and nuclearity, Riesz spectral theory, Fredholmness, ellipticity and Gohberg's lemma, among others, for pseudo-differential operators over the compact group of $p$-adic integers $\mathbb{Z}_p^d$, where the author in a recent paper introduced a notion of Hörmander classes and pseudo-differential calculus. We extend the results to compact Vilenkin groups which are essentially the same as $\mathbb{Z}_p^d$. Also we provide a new definition of H{\"o}rmander classes for pseudo-differential operators acting on non-compact Vilenkin groups and a explicit formula for the Fredholm spectrum in terms of the associated symbol.
\end{abstract}
\maketitle
 \tableofcontents
\section{\textbf{Introduction}}
The idea of representing linear operators with some ``less complex" mathematical object goes back centuries, and it has been studied in many important works. For example, the theory of pseudo-differential operators attemps to represent densely defined linear operators using operator valued symbols, and at the present, since the work of Kohn, Nirenberg, Hörmander among others to study problems in partial differential equations, many important properties have been put in terms of such symbol. A ``toy model" useful to understand the essentials of this technique, at least for compact smooth manifolds, is the $d$-dimensional torus $\T^d$. Linear operators acting on functions $f:\T^d \to \C$ can be described in terms of a complex valued function $\sigma: \T^d \times \Z^d \to \C$, called the symbol of the operator. More precisely, for a densely defined linear operator $A:D(A) \subseteq L^2 (\T^d) \to L^2 (\T^d)$ there exist a symbol $\sigma_A: \T^d \times \Z^d \to \C$ given by formula  $\sigma_A (x , k) := e^{- i x \cdot k} (A e_k ) (x),$ where $e_k (x):= e^{i x \cdot k}$, such that \begin{equation}\label{quantisation}
    (A f)(x) = \sum_{k \in \Z^d} \sigma_A (x , k) \widehat{f} (k) e^{i x \cdot k},
\end{equation}for every trigonometric polynomial $f \in Span\{e^{i x \cdot k }\}_{k \in \Z^d}$. Here, as usual, for an integrable function $f: \T^d \to \C$ we denote its $k$-th Fourier coefficient by $\widehat{f} (k)$.

A linear operator defined by formula (\ref{quantisation}) is often called a periodic pseudo-differential operator with symbol $\sigma_A$, and the most important example of such operators are the differential operators with smooth periodic coefficients. The interested reader may find more basic information about these linear operators and related topics in the reference books \cite{ruzhansky1, wongdfa}. With this symbolic representation, important properties such as boundedness \cite{LpboundsDuvan, LpboundsJulio}, compactness \cite{shahla}, belonging to Schatten classes and nuclearity \cite{Cardona2019, DELGADO2014779, nuclearJulio}, Riesz spectral theory \cite{Velasquez-Rodriguez2019}, Fredholmness, ellipticity and Gohber's lemma \cite{Molahajloo2010}, among others, have been characterised in terms of the symbol. 

Using the same ideas in a recent paper the author introduce a notion of pseudo-differential operators, Hörmander classes and symbolic calculus, for operators acting on measurable functions $f: \Z_p^d \to \C$, and discussed its relation with the quantisations given in \cite{pseudosvinlekin, harmonicfractalanalysis} for locally compact Vilenkin groups. Then, exactly in the same way as for $\T^d$, we can think on densely defined linear operators as linear operators given in terms of a symbol $\sigma:\Z^d_p \times \widehat{\Z}^d_p \to \C$ trough the formula \begin{equation}\label{pseudosZp}
    T_\sigma f (x) = \sum_{\xi \in \widehat{\Z}_p^d} \sigma (x , \xi) \widehat{f}(\xi) \chi_p (\xi \cdot x) .
\end{equation} Here $\chi_p (\xi \cdot x):= e^{2 \pi i \{\xi \cdot x\}_p}$ is a $p$-adic character, $\{ \cdot \}_p$ is the $p$-adic fractional part, and $$\widehat{f}(\xi):= \mathcal{F}_{\Z_p^d} [f] (\xi ) := \int_{\Z_p^d} f(x) \overline{ \chi_p (\xi \cdot x)} dx,$$ where $dx$ is the normalised Haar measure in $\Z_p^d$, is the Fourier transform in $\Z^d_p$. In this context the inverse Fourier transform is given by $$\mathcal{F}_{\widehat{\Z}_p^d} [\varphi](x) := \sum_{\xi \in \widehat{\Z}_p^d} \varphi (\xi) \chi_p (\xi \cdot x),$$for suitable functions $\varphi: \widehat{\Z}_p^d \to \C$. We will call a linear operator written in the form (\ref{pseudosZp}) a $p$-adic pseudo-differential operator with symbol $\sigma(x,\xi)$. 

Since all the analysis in the toroidal case relies only on the Fourier analysis of compact groups, we will be able to extend by analogy the theory from $\T^d$ to $\Z_p^d$. Moreover, although our analysis will be done for simplicity on $\Z_p$, most of it remains valid for general compact Vilenkin groups where, in the knowledge of the author, there is not a systematic treatment of the theory from the point of view of H{\"o}rmander classes, in contrast with the more  studied locally compact (non compact) case \cite{zunigagalindo2, Onneweer1978, pseudosvinlekin, harmonicfractalanalysis, zunigagalindo1}. Similar to $\R^d$, where a linear operator acting on smooth functions may be expressed as $$T f(x) = \int_{\R^d} \sigma_T (x , \xi) \widehat{f}(\xi) e^{i x \xi } d \xi,$$for local fields and locally compact Vilenkin groups there is a notion of pseudo-differential operators and symbol classes. For example, for locally compact Vilenkin groups, one may consider linear operators with the form $$T_\sigma f (x) = \int_{\widehat{G}} \int_G \sigma (x,\xi) f(y) \overline{\chi_G (\xi (y-x))} dy d\xi ,$$ where $\widehat{G}$ is the Pontryagin dual of the group $G$. Our goal here is to develop a systematic theory for the compact case, since the available literature seems to be focused in non-compact groups. More information about pseudo-differential operators over local fields may be found in the reference book \cite{harmonicfractalanalysis} and the references there in. 

The motivation for the development of a theory of pseudo-differential operators over Vilenkin groups comes mainly from two sources. The first is the theory of differentiation on a $p$-adic or a $p$-series Field, initiated by J.E. Gibss and M.J. Millard \cite{gibbs1, gibbs2} with their definition of differentiability in the dyadic group. For detailed treatment of this definition, its implications and applications, see \cite{Dyadic} and the references there in. Later P.L. Butzer and H.J. Wagner introduced a slight different definition to that of Gibss-Millar, and showed that with their definition of differentiability one obtain a theory similar to the classical theory of differentiability for functions defined on the torus, although there are also important differences between the two theories. This fact is remarked by C. W. Onneweer \cite{Onneweer1977, Onneweer1978} who also extends the definition to a $p$-adic or $p$-series field. The second motivation is the study of $p$-adic pseudo-differential equations on a ball, initiated with the study of Vladimirov operator on a $p$-adic ball in \cite{fisicapadica}, and later continued with the works of A. Kochubei \cite{kochubeibook, p-adicball1, p-adicball2} and  Avetisov and Bikulov \cite{2002J, Avetisov2014, 2009JPhA...42h5003A}, among others. Both motivations happen to be two different perspectives of the same situation, as they reader might notice from \cite[Theorem 1]{Onneweer1977} and \cite[Lemma 1]{basisQp}. The characters $\chi_p (\xi \cdot x)$ are eigenfunctions of the derivatives in the sense of \cite{Onneweer1977} and the Vladimirov operator in the sense of \cite{Avetisov2014}. Thus as linear operators they define the same Fourier multiplier (except maybe, depending of the definitions, for a constant) and serve as inducement to a theory for more general operators. Our goal now is to describe the basics of such a theory in $\Z_p^d$ and its generalisation to more general Vilenkin groups. 

Our exposition will be done as follows:
\begin{itemize}
    \item In Section 2 we recall some basic definitions and some results proven in \cite{p-adicHormanderclasses}.
    \item In Sections 3 we study the Sobolev boundedness of pseudo-differential operators.
    \item In Section 4 we give some necessary and sufficient conditions on the symbol of a pseudo-differential operator for belonging to the class of compact operators and the class of Riesz or inessential operators, similar to \cite{Velasquez-Rodriguez2019}.
    \item In Section 5 we study the nuclearity and summability of $s$-numbers for pseudo-differential operators.
    \item In Section 6 we provide an explicit formula for the Fredholm spectrum of a pseudo-differential operator in terms of the symbol.
    \item In Section 7 we prove a version of the Weyl law for the asymptotic growth of eigenvalues of sectorial $n$-hypoelliptic pseudo-differential operators. 
    \item In Section 8 we give a brief exposition about our notion of symbol classes in locally compact Vilenkin groups. 
\end{itemize}
\newpage
\section{\textbf{Preliminaries}}
Along this article $p$ will denote a fixed prime number. The field of $p$-adic numbers $\Q_p$ is defined as the complection of the field of rational numbers $\Q$ with respect to the $p$-adic norm $|\cdot|_p$ which is defined as \[|x|_p := \begin{cases}
0 & \esp \text{if} \esp x=0, \\ p^{-\gamma} & \esp \text{if} \esp x= p^{\gamma} \frac{a}{b},
\end{cases}\]where $a$ and $b$ are integers coprime with $p$. the integer $\gamma:= ord(x)$, with $ord(0) := + \infty$, is called the $p$-adic order of $x$. The unit ball of $\Q_p$ with the $p$-adic norm is called the compact group of $p$-adic integers and it will be denoted by $\Z_p$. Any $p$-adic number $x \neq 0$ has a unique expansion of the form $$x = p^{ord(x)} \sum_{j=0}^{\infty} x_j p^j,$$where $x_j \in \{0,1,...,p-1\}$ and $x_0 \neq 0$. By using this expansion, we define the fractional of $x \in \Q_p$, denoted by $\{x\}_p$, as the rational number\[\{x\}_p := \begin{cases}
0, & \esp \text{if} \esp x=0 \esp \text{or} \esp ord(x) \geq 0, \\ p^{ord(x)} \sum_{j=0}^{-ord(x)-1} x_j p^j,& \esp \text{if} \esp ord(x) <0.
\end{cases}\]We can extend the $p$-adic norm to $\Q_p^d$ by taking $$||x||_p := \max_{1 \leq i \leq d} |x_i|_p, \esp \esp \text{for } \esp x=(x_1,...,x_d) \in \Q_p^d.$$The unit ball of $\Q_p^d$ with the above defined norm will be denoted by $\Z_p^d$. It is known that $\Z_p^d$ is a compact totally disconected abelian group. Its dual group in the sense of Pontryagin, the collection of characters of $\Z_p^d$, will be denoted by $\widehat{\Z}_p^d$. The dual group of the $p$-adic integers is known to be the Pr{\"u}fer group $\Z (p^{\infty})$,  the unique $p$-group in which every element has $p$ different $p$-th roots. The Pr{\"u}fer group way be identified with the quotient group $\Q_p/\Z_p$. In this way the characteres of the group $\Z_p$ may be written as $$\chi_p (\xi  x) := e^{2 \pi i \{ \xi x \}_p}, \esp \esp x \in \Z_p, \esp \xi \in \widehat{\Z}_p \cong \Q_p / \Z_p.$$

By the Peter-Weyl theorem the elements of $\widehat{\Z}_p^d$ constitute an orthonormal basis for the Hilbert space $L^2 (\Z_p^d)$ which provide us a Fourier analysis for suitable functions defined on $\Z_p^d$ in such a way that the formula $$f(x) = \sum_{\xi \in \widehat{\Z}_p^d} \widehat{f}(\xi) \chi_p (\xi \cdot x),$$holds almost everywhere in $\Z_p^d$. Here $\widehat{f}$ denotes the Fourier transform of $f$ in turn defined as $$\widehat{f}(\xi):= \int_{\Z_p^d} f(x) \overline{\chi_p (\xi \cdot x)}dx,$$where $dx$ is the normalized Haar measure in $\Z_p^d$ and $\xi \cdot x := \xi_1 x_1 + ... + \xi_d x_d$. The above series are called the Fourier series of the function $f$.

Using the representation of functions in its Fourier series, for a given densely defined linear operator $$T: \hb^\infty := Span\{\chi_p (\xi \cdot x)\}_{\xi \in \widehat{\Z}_p^d } \subset D(T) \subseteq L^2 (\Z_p^d) \to L^2 (\Z_p^d),$$we can define its associated symbol $\sigma_T (x, \xi)$ by formula $$\sigma_T (x, \xi) = \overline{\chi_p (\xi \cdot x)} T \chi_p (\xi \cdot x).$$In this way we can think on any densely defined linear operator as a pseudo-differential operator given by formula $$Tf(x) = \sum_{\xi \in  \widehat{\Z}_p^d} \sigma_T (x,\xi) \widehat{f}(\xi) \chi_p (\xi \cdot x).$$
Our goal now is to characterize several properties of densely defined linear operators, which from now on will be be thought as pseudo-differential operators, in terms of its associated symbol. For that reason in what follows we will only consider pseudo-differential operators defined in terms of previously given a symbol. Specially, we will work with pseudo-differential operators whose symbols belong to the H{\"o}rmander classes given in Definition \ref{hormanderclasses}. These classes are thought in such a way that they include the Vladimirov operator in the ball given by the expression $$D^s f (x):= \sum_{\xi \in \widehat{\Z}_p^d} ||\xi||_p^s \widehat{f}(\xi) \chi_p (\xi \cdot x).$$See \cite{zunigagalindo2, p-adicball1, p-adicball2, pseudosvinlekin, zunigagalindo1} for more information about the Vladimirov-Taibleson operator and \cite{Onneweer1977, Onneweer1978, Dyadic, harmonicfractalanalysis} for its relation with the Gibbs derivative. Now we recall some definitions used in \cite{p-adicHormanderclasses}.
\begin{defi}\label{defp-adicsobolev}\normalfont
 For our purposes the $L^r$-based \emph{Sobolev spaces} $H^s_r (\Z_p^d)$ are the metric completion of $\hb^\infty :=Span\{\chi_p (\xi \cdot x)\}_{\xi  \in \widehat{\Z}_p^d},$ with the norm $$||f||_{H^s_r (\Z_p^d)} := ||J_s f||_{L^r (\Z_p^d)}.$$ Here we are using the notation $\langle \xi \rangle := \max\{1 , ||\xi||_p \},$ and $$J_s f (x):= \sum_{\xi \in \widehat{\Z}_p^d} \langle \xi \rangle^s \widehat{f}(\xi) \chi_p (\xi \cdot x).$$We will call \emph{smooth functions over} $\Z_p^d$ to the elements of the class$$C^\infty (\Z_p^d) := \bigcap_{s \in \R} H^s_2 (\Z_p^d).$$The class of \emph{distributions on} $\Z_p^d$ is defined as $$\mathfrak{D}(\Z_p^d) := \bigcup_{s \in \R} H^s_2 (\Z_p^d).$$ The class $C^{\infty} (\Z_p^d \times \Z_p^d)$ is defined as $$C^{\infty} (\Z_p^d \times \Z_p^d):=\Big\{g \in L^2 (\Z_p^d \times \Z_p^d) \esp : \Big|\esp \int_{\Z_p^d \times \Z_p^d} g(x,y) \overline{\chi_p(\xi \cdot x)} \overline{\chi_p(\eta \cdot y)} dy dx\Big| \leq C_{g, s_1,s_2 } \langle \xi \rangle^{s_1} \langle \eta \rangle^{s_2} \Big\},$$for all $s_1 , s_2 \in \R$. See \cite{sobolevmetrizablegroups} for more information about Sobolev spaces over metrizable groups. 
 \end{defi}
Also we recall the definition of H{\"o}rmander classes introduced in \cite{p-adicHormanderclasses} that we will use along this work.
\begin{defi}\normalfont\label{hormanderclasses}
    \esp
\begin{enumerate}
    \item[(i)] For functions $\varphi: \widehat{\Z}_p^d \to \C$ let us define the difference operator $\triangleplus$ as $$\triangleplus_\eta^\xi \varphi (\xi) := \varphi(\xi + \eta) - \varphi (\xi).$$
    \item[(ii)] Let $0 \leq \delta \leq \rho \leq 1$ be given real numbers. We define the H{\"o}rmander symbol classes $\Tilde{S}^m_{\rho , \delta} (\Z_p^d \times \widehat{\Z}_p^d)$, $m \in \R$, as the collection of measurable functions $\sigma: \Z_p^d \times \widehat{\Z}_p^d \to \C$ such that, for all $\alpha , \beta \in \N_0$, the estimate $$|D^\beta_x \triangleplus_\eta^\xi \sigma (x, \xi )| \leq C_{\sigma m \alpha \beta} ||\eta||^\alpha_p \langle \xi \rangle^{m - \rho \alpha + \delta 
    \beta} , $$holds for some $C_{\sigma m \alpha \beta}>0$, and any $(x, k,h) \in \Z_p^d \times \widehat{\Z}_p^d \times \widehat{\Z}_p^d.$ Furthermore, we define $$\Tilde{S}^{-\infty }(\Z_p^d \times \widehat{\Z}_p^d) := \bigcap_{m \in \R}  \Tilde{S}^m_{0, 0} (\Z_p^d \times \widehat{\Z}_p^d),$$ $$\Tilde{S}^\infty_{\rho, \delta} (\Z_p^d \times \widehat{\Z}_p^d) := \bigcup_{m \in \R} \Tilde{S}^m_{\rho, \delta} (\Z_p^d \times \widehat{\Z}_p^d).$$We call operators with symbol in the class $\Tilde{S}^{-\infty }(\Z_p^d \times \widehat{\Z}_p^d)$ \emph{smoothing operators}.
\end{enumerate}
\end{defi}
For operators with symbol in these H{\"o}rmander classes the following is proven in \cite[Section 7]{p-adicHormanderclasses}:
\medskip
\begin{pro}
Let $0 <\rho \leq 1$. Let $T_{\sigma_1}$, $T_{\sigma_2}$ be pseudo-differential operators with symbols in the Hörmander classes $\sigma_1 \in \Tilde{S}^{m_1}_{\rho , 0} (\Z_p \times \widehat{\Z}_p)$,  $\sigma_2 \in \Tilde{S}^{m_2}_{ \rho , 0} (\Z_p \times \widehat{\Z}_p)$. Then :
\begin{enumerate}
    \item[(i)] $T_\sigma T_\tau = T_{\sigma \tau} + R$, where $R \in Op({\Tilde{S}^{-\infty}} (\Z_p \times \widehat{\Z}_p))$ and we have $\sigma_1 \sigma_2 \in \Tilde{S}^{m_1 + m_2}_{\rho, 0} (\Z_p \times \widehat{\Z}_p)$.
    \item[(ii)] If $\sigma \in \Tilde{S}^{m}_{\rho , 0} (\Z_p \times \widehat{\Z}_p)$ then $T_\sigma^t , T_\sigma^* \in Op(\Tilde{S}^{m}_{\rho , 0} (\Z_p \times \widehat{\Z}_p))$ and $\sigma^* - \overline{\sigma}, \sigma^t - \sigma \in {\Tilde{S}^{-\infty}} (\Z_p \times \widehat{\Z}_p).$
\end{enumerate}
\end{pro}
To conclude this section we recall some well known definitions from functional analysis and spectral theory and introduce some notation.
\begin{defi}\normalfont\label{definotation}
Let $E,F$ be a Banach spaces and $T:D(T) \subset E \to F$ be a densely defined linear operator.
\begin{enumerate}
    \item[(i)] We will denote by $\mathcal{L}(E,F)$ the collection of bounded operators form $E$ to $F$. We use the notation $\mathcal{L}(E)$ when $E=F$. $\mathfrak{K}(E,F)$ will denote the collection of compact operators in $\mathcal{L}(E,F)$, and $\mathfrak{R}(E,F)$ will denote the component in $\mathcal{L}(E,F)$ of the radical of the ideal of compact operators in the sense of \cite{operatorsideals}, called the ideal of inessential operators. See \cite{Velasquez-Rodriguez2019}.
    \item[(ii)] For $T \in \mathcal{L}(E)$, $\lambda_k (T)$ will denote the $k$-th eigenvalue of $T$, and $s_k(T)$ its $k$-th $s$-number, for any given $s$-scale. When $E$ is a Hilbert space $s_k(T)$ denotes the $k$-th singular number of $T$.
    \item[(iii)] For a bounded linear operator $T \in \mathcal{L}(E, F)$ we denote by $Res (T)$ and $Res_F (T)$ the resolvent set of $T$ and the Fredholm resolvent set of $T$, respectively, in turn defined as $$Res(T):=\{\lambda \in \C \esp : \esp T- \lambda I \esp \text{is boundedly invertible} \},$$ $$Res_F (T):= \{\lambda \in \C \esp : \esp T- \lambda I \esp \text{is invertible modulo compact operators}\}.$$The Spectrum and the Fredholm spectrum of $T$ are then defined as $$Spec(T):=\C \setminus Res(T) \esp \esp \text{and} \esp \esp Spec_F (T):= \C \setminus Res_F (T). $$ 
    \item[(v)] Throughout this paper we will use sometimes the symbol ``$a \lesssim b$'' to indicate that the quantity ``$a$'' is less or equal than a certain constant times the quantity ``$b$''. 
\end{enumerate}
\end{defi}
\section{$L^r$-\textbf{Sobolev boundedness}}
One of the first natural questions one may ask is whether a pseudo-differential operator extends to a bounded operator between $L^r$-Sobolev spaces. For the $L^r$-boundedness a sufficient condition may be given exploiting the argument in \cite[Theorem 4.8.1]{ruzhansky1}, but for this it will be necessary a Sobolev embedding theorem and a $L^r$-multiplier theorem on $\Z_p$. Let us begin by giving an equivalent definition of Sobolev spaces in the present setting and by providing an embedding lemma. For more information about Sobolev spaces on locally compact metrizable groups see 
\cite{sobolevmetrizablegroups}.

\begin{rem}
Define the modified $L^r$-based Sobolev norm $||\cdot||_{H^s_r (\Z_p^d)}' \esp$ for $f \in H^s_r (\Z_p^d)$ as $$||f||_{H^s_r(\Z_p^d)}' := \Big( \int_{\Z_p^d} |D^s f(x)|^r dx \Big)^{1/r}< \infty.$$
Then the above norm is equivalent to the norm in Definition \ref{defp-adicsobolev} as a trivial consequence of the Littlewood-Paley theory summarised in \cite{LittPaleyQuek}. 
\end{rem}

\begin{lema}\label{Sobolevembedd}
Let $1 < r < \infty$ and $s > d/r$ be given real numbers. Let ${r'}$ be such that $\frac{1}{r} + \frac{1}{{r'}}=1$. Then $H^s_r (\Z_p^d)$ is contained in $L^\infty (\Z_p^d)$ with continuous inclusion.
\end{lema}

\begin{proof}
Just notice that by Hausdorff-Young inequality and the fact that $\Z_p^d$ has finite measure we have  $$||\widehat{f}||_{\ell^{{r'}} (\widehat{\Z}_p^d)} \leq || f||_{L^r (\Z_p^d)},$$for $1<r<\infty$. So we get  \begin{align*}
    ||f||_{L^\infty (\Z_p^d)} \leq ||\widehat{f}||_{\ell^1 (\widehat{\Z}_p^d)} &\leq \big( \sum_{\xi \in \widehat{\Z}_p^d} \langle \xi \rangle^{-sr} \big)^{1/r}||  \langle \xi \rangle^s \widehat{f}||_{\ell^{{r'}} (\widehat{\Z}_p^d)}\\ & \leq \big( \sum_{\xi \in \widehat{\Z}_p^d} \langle \xi \rangle^{-sr} \big)^{1/r} || J_s f||_{L^r (\Z_p^d)},
\end{align*}which conclude the proof.
\end{proof}

The above embedding lemma is enough to prove:

\begin{teo}\label{L2boundedness}
Let $T_\sigma$ be a pseudo-differential operator with symbol $\sigma (x , \xi)$, and let $s \in \R$ a given real number. Assume that $$||D^{\beta + |s|} \sigma (\cdot , \xi )||_{L^r (\Z_p^d)} \lesssim \langle \xi \rangle^{m},$$ for $ \beta >d/r$ and $r=2$. Then $T_\sigma$ extends to a bounded operator from $H^{s+m}_r (\Z_p^d)$ to $H^{s}_r (\Z_p^d)$.
\end{teo}

\begin{proof}
To begin with, $T_\sigma$ extends to a linear operator in $\mathcal{L}(H^{s+m}_2 (\Z_p^d) , H^{s}_2 (\Z_p^d))$ if and only if $A:=J_s T_\sigma J_{- (s+m)}$ extends to a bounded operator on $L^2 (\Z_p^d)$. The Symbol of $A$ is given by $$\sigma_A (x , \xi) = \frac{1}{\langle \xi \rangle^{s+m}}\sum_{\eta \in \widehat{\Z}_p^d} \langle \eta + \xi \rangle^{s} \widehat{\sigma} (\eta , \xi) \chi_p (\eta \cdot x). $$ Now let us define $$A_y f(x) := \sum_{\xi \in \widehat{\Z}_p} \sigma_A (y , \xi) \widehat{f}(\xi) \chi_p (\xi \cdot x ),$$ so $A_x f (x) = Af (x)$. Thus this give us $$||Af||_{L^2 (\Z_p^d)}^2 = \int_{\Z_p^d} |A_x f(x)|^2 dx \leq \int_{\Z_p^d} ||A_y f(x)||^2_{L^\infty_y (\Z_p^d)} dx,$$ which combined with the embedding theorem leads to $$ \int_{\Z_p^d} ||A_y f(x)||^2_{L^\infty_y (\Z_p^d)} dx \lesssim \int_{\Z_p^d} \int_{\Z_p^d} |D_y^\beta A_y f(x)|^2 dy dx.$$Changing the order of integration, which is valid because of Fubini's theorem, we obtain for $f \in Span\{\chi_p (\xi \cdot x)\}_{\xi \in \widehat{\Z}_p^d}$ \begin{align*}
    ||T_\sigma f||^2_{L^2 (\Z_p^d)} &\lesssim \int_{\Z_p^d} \int_{\Z_p^d} |D_y^\beta A_y f(x)|^2 dy dx \\ &=  \int_{\Z_p^d} \int_{\Z_p^d} |D_y^\beta A_y f(x)|^2 dx dy \\ &=  \int_{\Z_p^d}\Big( \sum_{\xi \in \widehat{\Z}_p^d } |D^\beta_y \sigma_A (y , \xi)|^2 |\widehat{f}(\xi)|^2\Big) dy \\ &= \sum_{\xi \in \widehat{\Z}_p^d } \int_{\Z_p^d} |D^\beta_y \sigma_A (y , \xi)|^2 dy |\widehat{f}(\xi)|^2.
\end{align*}
Observe that $$D^\beta_y \sigma_A (y , \xi) = \frac{1}{\langle \xi \rangle^{s+m}}\sum_{\eta \in \widehat{\Z}_p^d} \langle \eta \rangle^\beta \langle \eta + \xi \rangle^{s} \widehat{\sigma} (\eta , \xi) \chi_p (\eta \cdot y), $$ thus using Peetre inequality \begin{align*}
    ||D^\beta_y \sigma_A (y , \xi)||_{L^2_y (\Z_p^d)}^2 &= \frac{1}{\langle \xi \rangle^{2(s+m)}}\sum_{\eta \in \widehat{\Z}_p^d} \langle \eta \rangle^{2\beta} \langle \eta + \xi \rangle^{2s} |\widehat{\sigma}(\eta , \xi)|^2 \\ & \lesssim   \frac{1}{\langle \xi \rangle^{2m}} \sum_{\eta \in \widehat{\Z}_p^d} \langle \eta \rangle^{2(\beta+|s|)}  |\widehat{\sigma}(\eta , \xi)|^2 \\ &\lesssim  \frac{1}{\langle \xi \rangle^{2m}}||D_y^{\beta + |s|} \sigma (y, \xi)||_{L^2_y (\Z_p^d)}^2 \leq C.
\end{align*}Then finally $$||T_\sigma f||_{L^2 (\Z_p^d)} \lesssim ||f||_{L^2 (\Z_p^d)},$$concluding the proof.
\end{proof}
Of course the case $r \neq 2$ is quite more complicated and requires some kind of $L^r$-multiplier theorem. We will provide two sufficient conditions for $L^r$-boundedness using $\Lambda_r$-sets in the sense of \cite[Chapter 5]{HandbookV1} and the $L^r$-multiplier theorems given in \cite{LittPaleyQuek}.

\begin{defi}
A subset $\Lambda$ of the dual group $\widehat{\Z}_p^d$ is called a $\Lambda_r$-set, $1 <r<\infty $, if for some constant $ C_\Lambda $ and
all complex valued functions $\varphi: \widehat{\Z}_p^d \to \C$ we have the inequality$$\big| \big| \sum_{\xi \in \Lambda} \varphi (\xi) \chi_p (\xi \cdot x) \big| \big|_{L^r (\Z_p^d)} \leq C_\Lambda \Big( \sum_{\xi \in \Lambda} | \varphi (\xi) |^2 \Big)^{1/2}, \esp \esp \text{if} \esp r>2; $$  $$\big| \big| \sum_{\xi \in \Lambda} \varphi (\xi) \chi_p (\xi \cdot x) \big| \big|_{L^r (\Z_p^d)} \leq C_\Lambda\big| \big| \sum_{\xi \in \Lambda} \varphi (\xi) \chi_p (\xi \cdot x) \big| \big|_{L^1 (\Z_p^d)}, \esp \esp \text{if} \esp r \leq 2 $$
\end{defi} With the above definition an Theorem \ref{L2boundedness} we can prove the following sufficient condition for $L^r$-Sobolev boundedness.

\begin{coro}\label{LpboundednessSidon}
Let $T_\sigma$ be a pseudo-differential operator with symbol $\sigma(x , \xi)$. Let $\Lambda$ be a $\Lambda_r$-set for some $2<r<\infty$. Assume that $\sigma$ satisfies the hypothesis in Theorem \ref{L2boundedness}  and for all $x \in \Z_p$, $Supp \esp \sigma_x (\xi) \subseteq \Lambda$, where $\sigma_x (\xi) := \sigma (x,\xi)$. Then $T_\sigma$ extends to a bounded operator from $H^{s+m}_r (\Z_p^d)$ to $H^{s}_r (\Z_p^d)$.
\end{coro}

\begin{proof}
For $f \in Span\{\chi_p (\xi \cdot x) \}_{\xi \in \widehat{\Z}_p^d}$ define $A$, $\sigma_A$ and $A_y f (x)$ as in the proof of Theorem \ref{L2boundedness}. Using again Lemma \ref{Sobolevembedd} and Fubini's theorem we obtain \begin{align*}
    \int_{\Z_p^d} |T_\sigma f(x)|^r dx &\leq \int_{\Z_p^d } ||A_y f(x) ||_{L_y^\infty(\Z_p^d)}^r dx \\&\lesssim \int_{\Z_p^d} \int_{\Z_p^d} | D^{\beta}_y A_y f(x)|^r dy dx \\&\lesssim \int_{\Z_p^d} \int_{\Z_p^d} | D^{\beta}_y A_y f(x)|^r dx dy. 
\end{align*}Using $Supp \esp  \sigma_y (\xi) \subseteq \Lambda$ (our replacement of a suitable multiplier theorem) we get $$\int_{\Z_p^d} | D^{\beta}_y A_y f(x)|^r dx \leq C_\Lambda \Big( \sum_{\xi \in \Lambda } \Big |D^\beta_y \sigma_A (y , \xi)|^2 | \widehat{f} (\xi)|^2 \Big)^{r/2} \leq C_\Lambda \sup_{\xi \in \Lambda} |D^\beta_y \sigma_A (y , \xi)|^r ||f||^r_{L^2 (\Z_p^d)}.$$Finally, since $||f||_{L^2 (\Z_p^d)} \leq ||f||_{L^r (\Z_p^d)}$, the next inequality follows $$\int_{\Z_p^d} |T_\sigma f(x)|^p dx \lesssim \int_{\Z_p^d} \int_{\Z_p^d} | D^{\beta}_y A_y f(x)|^r dx dy \lesssim \sup_{\xi \in \Lambda} \int_{\Z_p^d} |D^\beta_y \sigma_A (y , \xi)|^r dy ||f||^r_{L^r (\Z_p^d)}, $$ concluding the proof. 
\end{proof}
The above corollary is an straigthfoward consequence of the main property of $\Lambda_r$ sets, that is, they make the $L^r$-norm an equivalent norm to the $L^2$-norm. In this way an $L^2$-estimative may be extended to an $L^r$-estimative. Another technique for treat the $L^r$-case trough $L^2$-methods is the well known Littlewood-Paley theory. It is a well developed theory in $\R^d$ and $\T^d$ that have been extended to the case of locally compact Vilekin groups in \cite{LittPaleyQuek}. Here we are interested mainly in \cite[Theorem 1]{LittPaleyQuek} which is a version of \cite{littlwoodpaley} adapted to Vilekin groups. For simplicity we state a version of \cite[Theorem 1]{LittPaleyQuek} for the compact group of $p$-adic integers.

\begin{lema}[T.S. Quek \cite{LittPaleyQuek}]\label{littlewoodpaley}
Let $1<r<\infty$ be a given real number. Then there exist positive constants $B_r , C_r$ such that $$B_r ||f||_{L^r (\Z_p^d)} \leq ||Sf||_{L^r (\Z_p^d)} \leq C_r ||f||_{L^r (\Z_p^d)}, \esp \esp \esp f \in C^\infty (\Z_p),$$where $$Sf(x):= \Big( \sum_{j \in \N} \big|\sum_{\xi \in \widehat{\Z}_p^d , \esp ||\xi||_p = p^j } \widehat{f}(\xi) \chi_p (\xi \cdot x) \big|^2\Big)^{1/2}.$$
\end{lema}
An immediate consequence of this equivalence of norms is the following multiplier theorem.
\begin{coro}\label{Lrmultiplier}
Let $T_\sigma$ be a Fourier multiplier with symbol $\sigma(\xi):=\sigma(||\xi||_p)$. Then $T_\sigma: L^r (\Z_p^d) \to L^r (\Z_p^d)$ extends to a bounded operator if and only if $\sigma \in \ell^\infty (\widehat{\Z}_p^d)$. In consecuence the spectrum of the Fourier multiplier $T_\sigma$ in $L^r(\Z_p^d)$ is given by $$Spec(T_\sigma) = \{\sigma(||\xi||_p) \esp : \esp \xi \in \widehat{\Z}_p^d \}.$$Moreover $T_\sigma$ defines an operator aproximmable by finite rank operators, and then a Compact operator in $L^r (\Z_p^d)$, if and only if $$\lim_{||\xi||_p \to \infty} |\sigma(||\xi||_p)| =0.$$
\end{coro}
And with the above version of Littlewood-Paley theorem and the same proof as in Theorem \ref{L2boundedness} we obtain:

\begin{coro}\label{lrboundedness}
Let $T_\sigma$ be a pseudo-differential operator with symbol $\sigma(x,||\xi||_p)$ and $s \in \R$ a given real number. Assume that $$||D^{\beta + |s|} \sigma (\cdot , ||\xi||_p )||_{L^r (\Z_p^d)} \lesssim \langle \xi \rangle^{
m},$$ for $1 < r < \infty$ and $ \beta >d/r$. Then $T_\sigma$ extends to a bounded operator from $H^{s+m}_r (\Z_p^d)$ to $H^{s}_r (\Z_p^d)$.  
\end{coro}

\begin{proof}
\begin{align*}
    \int_{\Z_p^d} |T_\sigma f(x)|^r dx &\lesssim \int_{\Z_p^d} \int_{\Z_p^d} | D^{\beta}_y A_y f(x)|^r dx dy \\& \lesssim \int_{\Z_p^d} \int_{\Z_p^d }\Big( \sum_{j \in \N} \big|\sum_{\xi \in \widehat{\Z}_p^d , \esp ||\xi||_p = p^j } D_y^\beta \sigma_A(y , ||\xi||_p) \widehat{f}(\xi) \chi_p (\xi \cdot x) \big|^2\Big)^{r/2} dx dy \\ & \lesssim \sup_{\xi \in \widehat{\Z}_p^d} \int_{\Z_p^d} |D_y^\beta \sigma_A(y , ||\xi||_p)|^r dy \int_{\Z_p^d }\Big( \sum_{j \in \N} \big|\sum_{\xi \in \widehat{\Z}_p^d , \esp ||\xi||_p = p^j }  \widehat{f}(\xi) \chi_p (\xi \cdot x) \big|^2\Big)^{r/2} dx \\ & \lesssim \sup_{\xi \in \widehat{\Z}_p^d} \int_{\Z_p^d} |D_y^\beta \sigma_A(y , ||\xi||_p)|^r dy||f||_{L^r (\Z_p^d)}^r.
\end{align*}
\end{proof}
To conclude this section we state the fact that operators in the H{\"o}rmander classes given in Definition \ref{hormanderclasses} extend to bounded operators between Sobolev spaces. This fact is easily deduced with a proof as the above and the multiplier theorem \cite[Theorem III.1]{pseudosvinlekin}.

\begin{teo}
Let $1<r < \infty$ be a given real number. Let $T_\sigma$ be a pseudo-differential operator with symbol $\sigma \in \Tilde{S}^m_{1 , 0} (\Z_p^d \times \widehat{\Z}_p^d)$ and let $s \in \R$ be a given real number. Then $T_\sigma$ extends to a bounded operator from $H^{s+m}_r (\Z_p^d)$ to $H^{s}_r (\Z_p^d)$.
\end{teo}

\section{\textbf{Compact and Inessential operators}} Another interesting property that can be put in terms of the symbol is the compactness of a pseudo-differential operator. In general, it is an interesting problem to put in terms of the symbol the belonging to a certain operator ideal for pseudo-differential operators in H{\"o}rmander classes. That is: to give conditions on the symbol in order to assure that a pseudo-differential operator $T_\sigma$ belongs to the intersection $$\mathcal{I} \cap Op(S^0_{\rho, \delta} (\Z_p^d \times \widehat{\Z}_p^d)).$$Here $\mathcal{I}$ denotes an operator ideal in the sense of \cite{operatorsideals, Pietsch}. 
\begin{defi}
{\normalfont Let $\mathcal{L}$ be the collection of all bounded operators defined on Banach spaces. That is

$$\mathcal{L} := \bigcup_{E,F \text{ Banach Spaces}} \mathcal{L} (E,F).$$

An operator ideal is a sub-collection $\mathcal{I}$ of $\mathcal{L}$ such that each one of its components

$$\mathcal{I} (E,F) := \mathcal{I} \cap \mathcal{L}, $$
satisfy the following conditions:

\begin{enumerate}
    \item[(i)] For each one-dimensional Banach space $E$ the identity map $I_E$ belongs to $\mathcal{I}(E)$.
    \item[(ii)] If $T_1 , T_2 \in \mathcal{I} (E,F)$ then $T_1 + T_2 \in \mathcal{I}(E,F)$.
    \item[(iii)] If $T \in \mathcal{I}(E,F)$ and $X\in \mathcal{L}(F,F_0) , Y \in \mathcal{L} (E_0 , E)$ then their composition $X  T Y \in \mathcal{I} (E_0 , F_0)$, where $E_0 $ and $F_0$ are Banach spaces.
    \end{enumerate}
    
If additionally each component $\mathcal{I} (E,F)$ is closed in the operator norm topology of $\mathcal{L} (E,F)$ then it is said that $\mathcal{I}$ is a closed operator ideal.}     
\end{defi}
In Sections 4 and 5 our main goal is to give conditions on the symbol $\sigma \in \Tilde{S}^0_{\rho, \delta} (\Z_p^d \times \widehat{\Z}_p^d)$ of a pseudo-differential operator for belonging to certain operator ideals. Our first step is the biggest closed ideal in general Banach spaces, the radical of the ideal of compact operators, called the ideal of inessential operators. In Hilbert spaces it is well known that the ideal of inessential operators coincide exactly with the ideal of compact operators. The ideal of compact operators $\mathfrak{K}(L^2(\Z_p^d))$ is the biggest closed operator ideal in $\mathcal{L}(L^2 (\Z_p^d))$ and, in general $L^r$-spaces, it is $\mathfrak{R}(L^r(\Z_p^d))$, the component of the ideal of inessential operators in $\mathcal{L} (L^r(\Z_p^d))$. In particular when $r=2$ the equality $\mathfrak{R}(L^2(\Z_p^d))=\mathfrak{K}(L^2(\Z_p^d))$ holds. 

For operators acting on $L^r (\Z_p^d)$ there exists a simple characterization of the ideal $\mathfrak{R}(L^r(\Z_p^d)) \cap Op(\Tilde{S}^0_{1,0} (\Z_p^d \times \widehat{\Z}_p^d))$ given in the following theorem:
\begin{teo}\label{uniform compactpseudos}
Let $T_\sigma \in Op(S^0_{1 , 0} (\Z_p^d \times \widehat{\Z}_p^d))$ be a pseudo-differential operator. Them $T_\sigma$ extends to bounded operator in the ideal $\mathfrak{R}(L^r(\Z_p^d)) \cap Op(\Tilde{S}^0_{1,0} (\Z_p^d \times \widehat{\Z}_p^d))$ if and only if $$d_\sigma :=\limsup_{||\xi||_p \to \infty} || \sigma (\cdot , \xi)||_{L^\infty (\Z_p^d)} = 0.$$
\end{teo}
The proof of the above theorem usually depends on the following lemma, usually called the Gohber's Lemma \cite{Molahajloo2010, Velasquez-Rodriguez2019}.

\begin{lema}\label{gohbers}
Let $1<r < \infty$ be a given real number. Let $r'$ be such that $\frac{1}{r} + \frac{1}{r'} = 1.$ Let $T_\sigma \in Op(\Tilde{S}^0_{1, 0} (\Z_p^d \times \widehat{\Z}_p^d))$ be a pseudo-differential operator. Then for every compact operator $K \in \mathfrak{K}(L^r (\Z_p^d))$ it holds $$||T_\sigma - K||_{\mathcal{L}(L^r (\Z_p^d))} \geq d_\sigma,$$ where $$d_\sigma := \limsup_{||\xi||_p \to \infty} ||\sigma (\cdot , \xi)||_{L^\infty (\Z_p^d)}.$$When $r=2$ the same holds true for $T_\sigma \in Op(\Tilde{S}^0_{0 , 0} (\Z_p^d \times \widehat{\Z}_p^d)).$
\end{lema}

\begin{proof}
\esp
We want to divide the proof in several steps:
\begin{enumerate}
    \item[(i)]Using the above embedding theorems in \cite{sobolevmetrizablegroups} we se that each function $\sigma (x , \xi)$ is a continuous function so actually$$||\sigma (\cdot , \xi)||_{L^\infty (\Z_p^d)}= \sup_{x \in \Z_p^d} |\sigma(x , \xi)|.$$Moreover, the sequence of functions $\{\sigma (x , \xi)\}_{\xi \in \widehat{\Z}_p^d}$ is equicontinuous. This because of the embedding of Sobolev spaces in H{\"o}lder spaces proven in \cite{sobolevmetrizablegroups} and the fact that all the derivatives of $\sigma$ are uniformly bounded. Explicitly, the following inequality holds true: $$|\sigma(x,\xi) - \sigma (y , \xi)| \leq |x - y|^{\alpha} ||\sigma (\cdot , \xi)||_{C^{0, \alpha} (\Z_p^d)} \lesssim |x - y|^{\alpha} ||\sigma (x , \xi )||_{H^{\alpha + d/2}_2 (\Z_p^d)}\lesssim |x - y|^{\alpha},$$for any $\alpha \in (0,1)$.
    \item[(ii)]The sequence $\{\chi_p (\xi_l \cdot x)\}_{l \in \N}$ converges to zero weakly. Hence for every compact operator $K$ the sequence $K \chi_p (\xi_l \cdot x)$ converges to zero strongly and we have for large $l$: $$||K \chi_p (\xi_l \cdot x) ||_{L^r(\Z_p^d)} < \varepsilon/2.$$Note that the same hold true if we replace $K$ for $K'= T K$ for any bounded operator $T$.
    \item[(iii)]$\Z_p^d$ is compact. Then for every every $\xi \in \widehat{\Z}_p^d$ there exists a $x_\xi \in \Z_p^d$ such that $$||\sigma (\cdot , \xi)||_{L^\infty (\Z_p^d)} = |\sigma (x_\xi , \xi)|.$$Let us take a sub-sequence $\{(x_{\xi_l} , \xi_l)\}$ of $\{(x_\xi , \xi)\}$ in such a way that $$\lim_{l \to \infty} |\sigma (x_{\xi_l} , \xi_l)| = d_\sigma.$$Because of the compactness of $\Z_p^d$ we can conclude that the sequence $\{x_{\xi_l}\}_{l \in \N}$ has an accumulation point $x_0$ and thus a sub sequence converging to $x_0$. Without loss of generality let us assume that $x_{\xi_l} \to x_0$ when $l \to \infty$. 
    \item[(iv)] Take a ball $B (x_0, \delta)$ in such a way that $$|\sigma(x , \xi) - \sigma(x_0 , \xi)| < \varepsilon/2,$$ for every $x \in B (x_0, \delta)$ and every $\xi \in \widehat{\Z}_p^d.$ Let $\psi_\delta$ be the characteristic function in $B (x_0, \delta)$, let $|B (x_0, \delta)|$ denote the volume of $B (x_0, \delta)$ and define $M_\delta$ as the multiplication operator $$M_\delta f(x):= \frac{\psi_\delta (x)}{|B (x_0, \delta)|^{1/\nu (r)}}  f (x),$$where $\nu (r):= \max\{2, 1 - 1/r = r'\}.$
    \item[(v)] For every $f \in L^r (\Z_p^d)$ we have the inequalities \begin{align*}
        ||T_\sigma - K||_{\mathcal{L}(L^r(\Z_p^d))} ||f||_{L^r(\Z_p^d)} &\geq ||M_\delta (T_\sigma - K)||_{\mathcal{L}(L^r(\Z_p^d))} ||f||_{L^r(\Z_p^d)} \\ & \geq ||M_\delta (T_\sigma - K) f||_{L^r(\Z_p^d)}\\ & \geq \big| ||M_\delta T_\sigma f  ||_{L^r(\Z_p^d)} - ||M_\delta K f ||_{L^r (\Z_p^d)}\big|.
    \end{align*}
    \item[(vi)] For large $l$ and $f= \chi_p (\xi_l \cdot x)$ we obtain \begin{align*}
        ||T_\sigma - K||_{\mathcal{L}(L^r(\Z_p^d))} \geq \Big( \int_{\Z_p^d} \frac{\psi_\delta (x)}{|B (x_0, \delta)|^{1/r'}}  |\sigma(x , \xi_l
    )|^{r} dx \Big)^{1/r} - \varepsilon/2 \geq |\sigma(x_{\xi_l} , \xi_l)| - \varepsilon
    \end{align*}Taking the limit as $l \to \infty$ the proof is complete. 
\end{enumerate}
\end{proof}

Using the above lemma and Corollary \ref{Lrmultiplier}, in the same spirit as \cite[Theorem 3.4]{Velasquez-Rodriguez2019}, we can prove a characterisation of compact operators in the Hörmander classes $S^0_{0,0} (\Z_p^d \times \widehat{\Z}_p^d)$ defined in \cite[Definition 4.8]{p-adicHormanderclasses}. The proof is the same as in \cite{Velasquez-Rodriguez2019} where the reader may refer for further details.

\begin{teo}
Let $T_\sigma$ be a pseudo-differential operator whose symbol $\sigma (x,||\xi||_p)$ satisfy the hypothesis of Corollary \ref{lrboundedness} for $s=m=0$. Then $T_\sigma$ defines a compact operator in $L^r (\Z_p^d)$ if and only if $$\limsup_{||\xi||_p \to \infty} ||\sigma(\cdot, ||\xi||_p)||_{L^\infty (\Z_p^d)} = 0. $$
\end{teo}
\begin{proof}
If $T_\sigma$ is a compact operator then, applying Lemma \ref{gohbers}, we get 
$$\limsup_{||\xi||_p \to \infty} ||\sigma (\cdot, ||\xi||_p)||_{L^\infty (\Z_p^d)} = 0.$$ Conversely, if the condition $$\limsup_{||\xi||_p \to \infty} ||\sigma(\cdot, ||\xi||_p)||_{L^\infty (\Z_p^d)} = 0 $$ holds consider the sequence of Fourier multipliers $\{ T_{\widehat{\sigma}_\eta} \}_{\eta \in \widehat{\Z}_p^d}$ defined trough the symbols $$\widehat{\sigma}_\eta (||\xi||_p) := \widehat{\sigma} (\eta , ||\xi||_p).$$ Each one defines a compact operator since $$\lim_{||\xi||_p \to \infty} |\widehat{\sigma} (\eta , ||\xi||_p)| \leq \lim_{||\xi||_p \to \infty} ||\sigma (\cdot, ||\xi||_p)||_{L^1 (\Z_p^d)}  \lesssim \limsup_{||\xi||_p \to \infty} ||\sigma (\cdot, ||\xi||_p)||_{L^\infty (\Z_p^d)} = 0,$$ so, the operators $$A_\eta f (x) := \chi_p (\eta \cdot x) T_{\widehat{\sigma}_\eta} f (x),$$ are compact, and the sum $$\sum_{||\eta||_p \leq N } A_\eta f (x) := \sum_{||\eta||_p \leq N} \chi_p (\eta \cdot x) T_{\widehat{\sigma}_\eta} f (x),$$is compact as well. These because compact operators form an ideal in $\mathcal{L} (L^r(\Z_p^d)).$ Moreover, this ideal is a closed in the operator norm topology, so if the limit $$\lim_{N\to \infty} \sum_{||\eta||_p \leq N } A_\eta f (x) := \lim_{N\to \infty}  \sum_{||\eta||_p \leq N} \chi_p (\eta \cdot x) T_{\widehat{\sigma}_\eta}  f(x),$$converges then the limit operator is compact in $\mathcal{L}(L^r (\Z_p^d))$. Finally notice that 
\begin{align*}
    \lim_{N\to \infty}  \sum_{||\eta||_p \leq N} \chi_p (\eta \cdot x) T_{\widehat{\sigma}_\eta} f 
    (x) &= \sum_{\eta \in \widehat{\Z}_p} \chi_p (\eta \cdot x) T_{\widehat{\sigma}_\eta} \\ &= \sum_{\eta \in \widehat{\Z}_p} \sum_{\xi \in \widehat{\Z}_p} \chi_p (\eta \cdot x) \widehat{\sigma} (\eta , ||\xi||_p) \widehat{f}(\xi) \chi_p (\xi \cdot x) \\&= \sum_{\xi \in \widehat{\Z}_p} \sum_{\eta \in \widehat{\Z}_p}  \chi_p (\eta \cdot x) \widehat{\sigma} (\eta , ||\xi||_p) \widehat{f}(\xi) \chi_p (\xi \cdot x) \\&= \sum_{\xi \in \widehat{\Z}_p} \sigma (x,||\xi||_p) \widehat{f}(\xi) \chi_p (\xi \cdot x) = T_\sigma f (x).
\end{align*}This concludes the proof.
\end{proof}
Now we can give a characterization of the ideal $\mathfrak{R}(L^r(\Z_p^d)) \cap Op(\Tilde{S}^0_{1,0} (\Z_p^d \times \widehat{\Z}_p^d))$ identical to \cite{Velasquez-Rodriguez2019}. 

\begin{proof}[Proof of Theorem \ref{uniform compactpseudos}]
The proof is the same as in \cite{Velasquez-Rodriguez2019}. If $T_\sigma \in \mathfrak{R}(L^r(\Z_p^d)) \cap Op(\Tilde{S}^0_{1,0} (\Z_p^d \times \widehat{\Z}_p^d))$ then $T_\sigma^k$ is a compact operator for some $k \in \N$. The symbol of $T_\sigma^k$ is $(\sigma)^k$ plus the symbol of an infinitely smoothing operator. By Gohberg's Lemma $d_{\sigma^k} = 0,$ and thus $d_\sigma = 0$. On the other hand if the condition $$\limsup_{||\xi||_p \to \infty} ||\sigma(\cdot, \xi)||_{L^\infty (\Z_p^d)} = 0 $$ holds consider the sequence of Fourier multipliers $\{ T_{\widehat{\sigma}_\eta} \}_{\eta \in \widehat{\Z}_p^d}$ defined trough the symbols $$\widehat{\sigma}_\eta (\xi) := \widehat{\sigma} (\eta , \xi).$$They are bounded operators because of \cite[Theorem  III.1]{pseudosvinlekin}. Each one belongs to $\mathfrak{R}(L^r(\Z_p^d)) \cap Op(\Tilde{S}^0_{1,0} (\Z_p^d \times \widehat{\Z}_p^d))$ since $$\lim_{||\xi||_p \to \infty} |\widehat{\sigma} (\eta , \xi)| \lesssim \limsup_{||\xi||_p \to \infty}   ||\sigma (\cdot, \xi)||_{L^\infty (\Z_p^d)} = 0,$$ see \cite[Remark 2.2]{Velasquez-Rodriguez2019}. So, the operators $$A_\eta f (x) := \chi_p (\eta \cdot x) T_{\widehat{\sigma}_\eta} f (x),$$ belong to the ideal $\mathfrak{R}(L^r(\Z_p^d)) \cap Op(\Tilde{S}^0_{1,0} (\Z_p^d \times \widehat{\Z}_p^d))$, and the sum $$\sum_{||\eta||_p \leq N } A_\eta f (x) := \sum_{||\eta||_p \leq N} \chi_p (\eta \cdot x) T_{\widehat{\sigma}_\eta} f (x),$$is also in $\mathfrak{R}(L^r(\Z_p^d)) \cap Op(\Tilde{S}^0_{1,0} (\Z_p^d \times \widehat{\Z}_p^d))$. This concludes the proof.
\end{proof}
Notice that for $\sigma \in \Tilde{S}^0_{0,0} (\Z_p^d \times \widehat{\Z}_p^d)$ when $r=2$ we don't need a multiplier theorem. Thus, with the same proof, we obtain:
\begin{coro}
Let $T_\sigma \in Op(S^0_{0 , 0} (\Z_p^d \times \widehat{\Z}_p^d))$ be a pseudo-differential operator. Then $T_\sigma$ extends to compact operator on $L^2(\Z_p^d)$ if and only if $$d_\sigma :=\limsup_{|\xi|_p \to \infty} || \sigma (\cdot , \xi)||_{L^\infty (\Z_p^d)} = 0.$$
\end{coro}

\section{\textbf{Nuclearity, Schatten-Von Neumann classes and Singular numbers}}

In this section we continue investigating the relation between operator ideals and the symbol of a pseudo-differential operator. We prove necessary and sufficient conditions on the symbol for belonging to certain operator ideals. To begin we need to remember the definition of some operator ideals well known in the literature.

\begin{defi}\normalfont
\esp 
\begin{enumerate}
    \item[(i)]Let $E,F$ be Banach spaces. We say that a linear operator $T \in \mathcal{L}(E,F)$ is $\gamma$-\emph{nuclear}, $0<\gamma \leq 1$, if there exist a so called $\gamma$-nuclear representation \begin{equation}\label{nuclear}
    T= \sum_{j =1}^\infty a_j \otimes y_j
    \end{equation} where $a_j \in E'$ and $y_j \in F$ for $j \in \N$, and $\{||a_j||_{E'} ||y_j||_{F} \}_{j \in \N} \in \ell^\gamma(\N)$. The collection of $\gamma$-nuclear operators in $\mathcal{L}(E,F)$ will be denoted by $\mathfrak{N}_\gamma (E,F)$.
    \item[(ii)] Let $H_1 , H_2$ be Hilbert spaces. Let $T \in \mathcal{L}(H_1 , H_2)$. We say that $T$ belongs to the Schatten-von Neumann class $S_{\gamma} (H_1 , H_2)$, $1 \leq \gamma < \infty$, if $$||T||_{S_\gamma (H_1 , H_2)} := \Big( \sum_{k \in \N_0}s_k (T)^\gamma \Big)^{1/\gamma}< \infty .$$
    \item[(iii)] We say that a linear operator $T \in \mathcal{L}(H_1 , H_2)$ is in the \emph{Dixmier class} $\mathcal{L}^{(1 , \infty)} (H_1 , H_2)$ if $$||T||_{\mathcal{L}^{(1 , \infty)}(H_1 , H_2)} := \sup_{N \geq 1} \frac{1}{\log (1 + N)} \sum_{0 \leq k \leq N} s_k(T)< \infty.$$
    \item[(iv)] Let $E,F$ be banach spaces. A sequence $\{s_k\}_{k\in \mathbb{N}}$ of functions that assings to every $T\in \mathcal{L}(E,F)$ a non-negative real number $s_k(T)$ with the following properties
\begin{enumerate}
\item[(s1)] $\norm{T}=s_1(T) \geq s_2(T) \geq ... \geq 0$ for all $T\in \mathcal{L}(E,F)$.
\item[(s2)] $s_{m+n-1}(T+R) \leq s_m(T) + s_n(R)$ for every $T,R\in \mathcal{L}(E,F)$.
\item[(s3)] $s_k(XTY) \leq \norm{X} s_k(T) \norm{Y}$ for every $Y\in \mathcal{L}(E_0,E)$, $T\in \mathcal{L}(E,F)$, $X\in \mathcal{L}(F,F_0)$.
\item[(s4)] $s_k(T)=0$ if $T\in \mathcal{L}(E,F)$ and $rank(T)< k$.
\item[(s5)] $s_k(I_k)=1$ where $I_k : \mathbb{C}^k \longrightarrow \mathbb{C}^k$ is the identity function. 
\end{enumerate}
is called a sequence a sequence of $s$-numbers. Some well known examples are:

\begin{itemize}
\item The approximation numbers: 
\begin{align*}
a_n(T):= inf\{\norm{T-A} \text{  } | \text{  } A\in \mathcal{L}(E,F), rank(T)<n\}
\end{align*}
\item The Weyl numbers: 
\begin{align*}
x_n(T):=sup\{a_n(TA) \text{  }| \text{  } A\in \mathcal{L}(l_2,E), \norm{A} \leq 1 \}
\end{align*}
\item The Hilbert numbers: 
\begin{align*}
h_n(T):=sup\{a_n(BTA) \text{  }| \text{  } A\in \mathcal{L}(l_2,E),B\in \mathcal{L}(F,l_2) \norm{A},\norm{B} \leq 1 \}
\end{align*}
\end{itemize}
\item[(v)] We say that $T \in \mathcal{L}(H_1 , H_2)$ is of $(r,s)$-type if $$||T_\sigma||_{\mathfrak{L}_{s,w} (L^2 (\Z_p^d))} := \Big( \sum_{k \in \N_0} [k^{\frac{1}{r} - \frac{1}{w}} s_k (T)]^w \Big)^{1/w}< \infty.$$
\end{enumerate}
\end{defi}
Notice that pseudo-differential operators are expressed by default in a similar way as in (\ref{nuclear}) so, if we impose some properties on the symbol then we can assure the $\gamma$-nuclearity. The following might  be consider an analogue of \cite[Theorem 3.1]{nuclearJulio}.

\begin{teo}
Let $T_\sigma:L^{r_1} (\Z_p^d) \to L^{r_2} (\Z_p^d)$ be a pseudo-differential operator with symbol $\sigma$ such that $$\sum_{\xi \in \widehat{\Z}_p^d} ||\sigma (\cdot , \xi)||_{L^{r_2} (\Z_p^d)}^{\gamma} < \infty.$$Then $T_\sigma$ is a $\gamma$-nuclear operator.
\end{teo}

In the case of the Hilbert space $L^2 (\Z_p)$, for operators in the Hörmander classes $\Tilde{S}^0_{\rho, 0} (\Z_p^d \times \widehat{\Z}_p^d)$, we can prove a necessary and sufficient condition for belonging to the Schatten-Von Neuman classes. 

\begin{teo}
Let $0<\rho \leq 1$. Let $T_\sigma$ be a pseudo-differential operator with symbol $\sigma \in \Tilde{S}^0_{\rho, 0} (\Z_p^d \times \widehat{\Z}_p^d)$. Then $T_\sigma$ extends to a bounded operator in the Schatten-Von Neumann class $S_\gamma (L^2 (\Z_p^d))$, $1\leq \gamma < \infty$, if and only if $$\sum_{\xi \in \widehat{\Z}_p^d} || \sigma (\cdot, \xi) ||_{L^\gamma (\Z_p^d) }^\gamma < \infty.$$  
\end{teo}

\begin{proof}
First notice $|T_\sigma|^2 := T_\sigma^* T_\sigma = T_{|\sigma|^2} + R_1$. Also $(|T_\sigma|^2)^n = T_{|\sigma|^{2n}} + R_2$ for every $n \in \N$, where $R_1, R_2 \in \Tilde{S}^{-\infty} (\Z_p^d \times \Z_p^d)$. From this we deduce that $f(|T_\sigma|^2) = T_{f(|\sigma|^2)} + R_f$, $R_f \in \Tilde{S}^{-\infty} (\Z_p^d \times \Z_p^d)$ for every holomorphic function, in particular for $f(z)=z^{\gamma}$, where the non integer powers are defined in terms of the complex logarithm. See \cite[Remark 7.2]{p-adicHormanderclasses}. Thus $|T_\sigma|^\gamma = T_{|\sigma|^\gamma} + R_\gamma$ for some $R_\gamma \in \Tilde{S}^{-\infty} (\Z_p^d \times \Z_p^d)$. Finally $$Tr(|T_\sigma|^\gamma)= \sum_{\xi \in \widehat{\Z}_p^d} ( |T_\sigma|^{\gamma} \chi_p (\xi \cdot x) , \chi_p (\xi \cdot x) )_{L^2 (\Z_p^d)} = \sum_{\xi \in \widehat{\Z}_p^d} \int_{\Z_p^d} [ | \sigma (x, \xi)|^\gamma + \sigma_{R_\gamma} (x , \xi) ] dx,$$concluding the proof.   
\end{proof}
\begin{coro}\label{corohilbertschmitpseudos}
Let $T_\sigma$ be a pseudo-differential operator with symbol $\sigma \in \Tilde{S}^0_{\rho, 0} (\Z_p^d \times \widehat{\Z}_p^d)$. Then $T_\sigma$ extends to a Hilbert-Schmit operator if and only if $$||T_\sigma||_{HS}^2 = \sum_{\xi \in \widehat{\Z}_p^d} || \sigma (\cdot, \xi) ||_{L^2 (\Z_p^d) }^2 < \infty.$$
\end{coro}
Now we turn our attention to individual singular values of pseudo-differential operators. In H{\"o}rmander classes these numbers may be estimated in terms of the associated symbol thanks to a simple but powerful tool:

\begin{pro}\label{proinftynorm}
Let $T_\sigma \in Op(\Tilde{S}^0_{0,0}(\Z_p^d \times \widehat{\Z}_p^d))$ be a pseudo-differential operator. Define on the algebra $Op(S^0_{0,0}(\Z_p^d \times \widehat{\Z}_p^d))$ the norms $||\cdot ||_{L^\infty}$ and $||\cdot ||_{\ell^\infty}$ by $$||T_\sigma ||_{L^\infty} := \sup_{\xi \in \widehat{\Z}_p^d} ||\sigma(x , \xi)||_{L^\infty (\Z_p^d)},$$ and $$|| T_\sigma ||_{\ell^\infty} := \sup_{\xi \in \widehat{\Z}_p^d} \sup_{\eta \in \widehat{\Z}_p^d} |\widehat{\sigma} (\eta , \xi)|.$$Then  $||T_\sigma||_{L^\infty} \leq ||T_\sigma||_{\mathcal{L}(L^2(\Z_p^d))} \lesssim  ||T_\sigma||_{L^\infty}$.
\end{pro}
During the proof of the above proposition we will the notation of \cite[Section 6]{p-adicHormanderclasses}.
\begin{proof}
Take an $x_\xi \in \Z_p^d$ such that $|\sigma(x_\xi , \xi)|= ||\sigma(\cdot , \xi)||_{L^\infty (\Z_p^d)}$. Take a ball $B (x_\xi , \delta(\xi))$ small enough around $x_\xi$. Define the multiplication operator $$M_\xi f (x):= \frac{\psi_\xi (x)}{|B (x_\xi , \delta(\xi))|^{1/2}} f(x),$$where $|B (x_\xi , \delta(\xi))|$ is the measure of $B (x_\xi , \delta(\xi))$. Then: \begin{align*}
    ||T_\sigma||_{\mathcal{L}(L^2(\Z_p^d))} &\geq ||M_\xi T_\sigma||_{\mathcal{L}(L^2(\Z_p^d))}\\ & \geq||M_\xi T_\sigma \chi_p (\xi \cdot x)||_{L^2 (\Z_p^d)} \\ &= \Big( \int_{\Z_p^d} \frac{\psi_\xi (x)}{|B (x_\xi , \delta(\xi))|} |\sigma(x , \xi)|^2 dx \Big)^{1/2} \\ & \geq ||\sigma(\cdot , \xi)||_{L^\infty (\Z_p^d)} - \varepsilon, 
\end{align*}for all $\varepsilon>0$ and every $\xi \in \widehat{\Z}_p^d$. In the opposite direction just note that the collection of infinite matrices $M_\sigma$ associated to an operator in $Op(\Tilde{S}^0_{0,0}(\Z_p^d \times \widehat{\Z}_p^d))$, which is the intersection of all the Schur classes as it is explained in \cite[Section 6]{p-adicHormanderclasses}, equals the class $\mathcal{D}$ in \cite[Theorem 12.6.9]{ruzhansky1}. Hence the following holds:
$$||T_\sigma||_{\mathcal{L}(L^2(\Z_p^d))} \lesssim ||T_\sigma||_{\ell^\infty} \leq ||T_\sigma||_{L^\infty}.$$
\end{proof}

\begin{coro}\label{coroequivopL2norm}
Let $T_\sigma \in Op(\Tilde{S}^0_{0,0}(\Z_p^d \times \widehat{\Z}_p^d))$ be a pseudo-differential operator. Then $$ \sup_{\xi \in \widehat{\Z}_p^d} ||\sigma(\cdot , \xi)||_{L^2 (\Z_p^d)} \leq ||T_\sigma||_{\mathcal{L}(L^2(\Z_p^d))} \lesssim \sup_{\xi \in \widehat{\Z}_p^d} ||\sigma(\cdot , \xi)||_{L^2 (\Z_p^d)}.$$
\end{coro}
Proposition \ref{proinftynorm} and its corollary are a convenient change of norms in the H{\"o}rmander class $Op(\Tilde{S}^0_{0,0}(\Z_p^d \times \widehat{\Z}_p^d))$ that will serve us to estimate the singular values of $T_\sigma$. Before continuing we need to make an important remark.
\begin{rem}\label{remdistancecompact}
The following equality holds true for every $T_\sigma \in Op(\Tilde{S}^0_{0,0}(\Z_p^d \times \widehat{\Z}_p^d))$:$$\inf_{K \in \mathfrak{K}(L^2 (\Z_p^d))} ||T_\sigma - K||_{\mathcal{L}(L^2 (\Z_p^d))} = \inf_{K \in \Tilde{S}^{-\infty}(\Z_p^d \times \widehat{\Z}_p^d) } ||T_\sigma - K||_{\mathcal{L}(L^2 (\Z_p^d))} .$$The reason is simple. First, every compact operator is approximated by a finite rank operator, let us write it as $$K_n f (x) = \sum_{j \leq n} c_j (f , g_j)_{L^2 (\Z_p^d)} h_j(x),$$ where $g_j, h_j \in L^2(\Z_p^d)$ for $0 \leq j \leq n$. And second, every function in $L^2 (\Z_p^d)$ is approximated by a smooth function. Thus, for every compact operator $K$ we can always find an $n \in \N_0$ and $g_1',..,g_n',h_1',...,h_n' \in C^\infty (\Z_p^d)$such that $K$ is approximated by $$K_n'f (x) = \sum_{j \leq n} c_j (f , g_j')_{L^2 (\Z_p^d)} h_j'(x).$$ The finite rank operator $K_n'$ is bounded between any pair of Sobolev spaces and thus it belong to $Op(\Tilde{S}^{-\infty}(\Z_p^d \times \widehat{\Z}_p^d)).$ See \cite[Proposition 4.14]{p-adicHormanderclasses}.
\end{rem}
\begin{lema}\label{lemmasvalus}
Let $T_\sigma \in Op(\Tilde{S}^0_{0,0}(\Z_p^d \times \widehat{\Z}_p^d))$ be a compact operator. Give an order to the set $\widehat{\Z}_p^d$, let us say $\widehat{\Z}_p^d = \{\xi_k\}_{k \in \N_0}$ in such a way that $$l_1 \leq l_2 \implies ||\sigma(x , \xi_{l_1})||_{L^2 (\Z_p^d)} \geq ||\sigma(x , \xi_{l_2})||_{L^2 (\Z_p^d)} .$$ Then $$|| \sigma(x , \xi_k) ||_{L^2 (\Z_p^d)} \leq  s_k (T_\sigma) \lesssim || \sigma(x , \xi_k) ||_{L^2 (\Z_p^d)}.$$
\end{lema}
\begin{proof}
There is only one $s$-scale on Hilbert spaces so, the singular values $s_k$ of a linear operator acting on Hilbert spaces are exactly equal to the approximation numbers $a_k$. Thus we will estimate the approximation numbers. From the definition of approximation numbers is obvious that $$a_k(T_\sigma) \leq ||T_\sigma - L_k ||_{\mathcal{L}(L^2(\Z_p^d))} \lesssim \sup_{l \geq k} ||\sigma(\cdot , \xi_l)||_{L^2 (\Z_p^d)} = || \sigma(x , \xi_k) ||_{L^2 (\Z_p^d)}  ,$$ where $L_k$ is the finite rank operator with symbol $\sigma (x , \xi_l)$ for $l <k$ and zero otherwise; $L_0 :=0$. Conversely, if $P_k$ is a finite rank operator with $rank(P_k) \leq k$, then consider the collection of $k+1$ linearly independent functions $\{\chi_p (\xi_0 \cdot x) , ..., \chi_p(\xi_k \cdot x)\}$. It is known that $P_k$ may be written as $$P_k f (x) = \sum_{j=0}^{k-1} c_j ( \varphi_j , f)_{L^2 (\Z_p^d)} \psi_j (x) ,$$ and it is obvious that at least one of the functions in $\{\chi_p (\xi_0 \cdot x) , ..., \chi_p(\xi_k \cdot x)\}$ is not in $Span\{\varphi_j \}_{0 \leq j \leq k-1},$ and thus is in its orthogonal complement. Let us say that this function is $\chi_p (\xi_{k_0} \cdot x)$. Hence $$||T_\sigma - P_k||_{\mathcal{L}(L^2 (\Z_p^d))} \geq ||(T_\sigma - P_k) \chi_p (\xi_0 \cdot x)||_{L^2 (\Z_p^d)} = ||\sigma (\cdot , \xi_{k_0})||_{L^2 (\Z_p^d)} \geq || \sigma(x , \xi_k) ||_{L^2 (\Z_p^d)}.$$This conclude the proof. 
\end{proof}

With Lemma \ref{lemmasvalus} we are now ready to classify some sub-ideals of  $Op(\Tilde{S}^0_{0,0}(\Z_p^d \times \widehat{\Z}_p^d))$ in terms of the symbol.

\begin{teo}
Let $T_\sigma \in Op(\Tilde{S}^0_{0,0}(\Z_p^d \times \widehat{\Z}_p^d))$ be a compact pseudo-differential operator. Assume that the sequence $\{||\sigma (x , \xi_k)||_{L^2 (\Z_p^d)}\}_{k \in \N_0}$ is ordered in a non increasing order. Then:
\esp 
\begin{enumerate}
    \item[(i)] $T_\sigma$ is a Dixmier traceable if and only if $$||T_\sigma||_{\mathcal{L}^{(1 , \infty)} (L^2 (\Z_p^d))}:= \sup_{N \geq 1} \frac{1}{\log(1 + N)} \sum_{0 \leq k \leq N} ||\sigma (\cdot , \xi_k)||_{L^2 (\Z_p^d)} < \infty.$$
    \item[(ii)] $T_\sigma$ is of $(s , w)$-type if and only if $$||T_\sigma||_{\mathfrak{L}_{s,w} (L^2 (\Z_p^d))} := \Big( \sum_{k \in \N_0} [k^{\frac{1}{r} - \frac{1}{w}} ||\sigma (\cdot , \xi_k)||_{L^2 (\Z_p^d)}]^w \Big)^{1/w}.$$
    \item[(iii)] $T_\sigma$ belong to the Schatten-von Newmann class $S_\gamma (L^2 (\Z_p^d))$ if and only if $$||T_\sigma||_{S_\gamma (L^2 (\Z_p^d))} := \Big( \sum_{k \in \N_0}||\sigma (\cdot , \xi_k)||_{L^2 (\Z_p^d)}^\gamma \Big)^{1/\gamma}< \infty .$$
\end{enumerate}
\end{teo}

\begin{rem}
The proof of the above theorem does not use anything special of the compact abelian group $\Z_p^d$. Thus the same proof works on more general groups. For example the statement of the above theorem holds true if we replace our $p$-adic H{\"o}rmander classes with the periodic classes $S^0_{0,0} (\T^d \times \Z^d)$.
\end{rem}

\section{\textbf{The Fredholm spectrum}}
The purpose of this section is to provide an explicit formula for the Fredholm spectrum of a pseudo-differential operator. Since the linear spaces that we we will consider here are only $L^2$-based Sobolev spaces, and the linear operators are all pseudo-differential operators with symbol in a H{\"o}rmaner class, we will modify subtly the definition of Fredholm resolvent set and Fredholm spectrum. This modification is valid and yields to the same spectrum for pseudo-differential operators because of the equivalence between Fredholmness and ellipticity in H{\"o}rmander classes proven in \cite[Theorem 8.11]{p-adicHormanderclasses}.
\begin{defi}\normalfont\label{Adaptedfredholmdefi}
Let $s , t \in \R$. Let $T\in \mathcal{L} (H^s_2 (\Z_p^d), H^t_2 (\Z_p^d))$ be a linear operator. We will say that $T$ is a Fredholm operator if it is invertible modulo an smoothing operator, that is, if there exist a linear operator $T^\bot \in \mathcal{L}(H^t_2 (\Z_p^d) , H^s_2 (\Z_p^d))$ such that $$T^\bot T - I , T T^\bot - I \in Op(S^{-\infty} (\Z_p^d \times \widehat{\Z}_p^d)).$$Consequently, when $s\geq t$ we define $$Res_F (T):= \{\lambda \in \C \esp : \esp T- \lambda \iota_{s,t} \esp \text{is invertible modulo} \esp Op(S^{-\infty} (\Z_p^d \times \widehat{\Z}_p^d)) \}.$$Here $\iota_{s,t}:H^s_2 (\Z_p^d) \to H^t_2 (\Z_p^d)$ is the inclusion operator. When $s=t$, $\iota_{s,s} = I$ is the identity operator.
\end{defi}
Using the symbolic calculus in \cite{p-adicHormanderclasses} we can provide an explicit formula for the Fredholm spectrum of a pseudo-differential operator in terms of its associated symbol.

\begin{teo}\label{Fredholmspectrumformula}
Let $0 < \rho \leq 1$ be a real number. Let $T_\sigma$ be a pseudo-differential operator with associated symbol in the H{\"o}rmander class $\Tilde{S}^0_{\rho,0} (\Z_p^d \times \widehat{\Z}_p^d)$. Then, considering $T_\sigma$ as a bounded operator on $L^2(\Z_p^d)$, we have$$Spec_F (T_\sigma) = \bigcap_{n \in \N_0} \overline{\bigcup_{||\xi||_p \geq n} \{\sigma (x,\xi) \esp : \esp x \in \Z_p^d\}}. $$
\end{teo}
\begin{proof}
We will prove that $$Res_F(T_\sigma) = \bigcup_{n \in \N_0} \C \setminus \overline{\bigcup_{||\xi||_p \geq n} \{\sigma (x,\xi) \esp : \esp x \in \Z_p^d\}}.$$For doing this just take a complex number $\lambda$ in the following set: $$\lambda \in \C \setminus \overline{\bigcup_{||\xi||_p \geq n_0} \{\sigma (x,\xi) \esp : \esp x \in \Z_p^d\}}, $$for some $n_0 \in \N$. Clearly $T_{\sigma - \lambda} \in Op(\Tilde{S}^{0}_{\rho,0} (\Z_p^d \times \widehat{\Z}_p^d))$ is a elliptic operator in the sense of \cite[Definition 8.3]{p-adicHormanderclasses} and by \cite[Lemma 8.7]{p-adicHormanderclasses} the symbol 
\[\beta(x,\xi) := \begin{cases}
\frac{1}{\sigma(x,\xi) - \lambda} & \esp \text{if} \esp ||\xi||_p \geq n_0, \\ 0 & \esp \text{if} \esp ||\xi||_p <n_0,
\end{cases}\]belongs to the H{\"o}rmander class $\Tilde{S}^0_{0,0} (\Z_p^d \times \widehat{\Z}_p^d)$. Moreover, $T_\beta$ is a an inverse of $T_\sigma$ modulo an smoothing operator and then $\lambda \in Res_F(T_\sigma)$. This proves $$ \bigcup_{n \in \N_0} \C \setminus \overline{\bigcup_{||\xi||_p \geq n} \{\sigma (x,\xi) \esp : \esp x \in \Z_p^d\}} \subset Res_F(T_\sigma),$$hence $$Spec_F (T_\sigma) \subseteq \bigcap_{n \in \N_0} \overline{\bigcup_{||\xi||_p \geq n} \{\sigma (x,\xi) \esp : \esp x \in \Z_p^d\}}. $$Conversely, we can use \cite[Theorem A.1.3]{fredholmalg} and the fact that $Op(\Tilde{S}^0_{\rho,0} (\Z_p^d \times \widehat{\Z}_p^d))$ is a $*$-subalgebra of the $C^*$-algebra $\mathcal{L}(L^2 (\Z_p))$ to assure that if $$T^\bot (T_\sigma-\lambda I) - I , (T_\sigma-\lambda I )T^\bot - I \in Op(S^{-\infty} (\Z_p^d \times \widehat{\Z}_p^d)),$$ then $T^\bot = T_\beta \in Op(\Tilde{S}^0_{\rho,0} (\Z_p^d \times \widehat{\Z}_p^d))$ and in this way $$(\sigma (x,\xi) - \lambda) \beta (x,\xi) - 1 \in S^{-\infty} (\Z_p^d \times \widehat{\Z}_p^d).$$Hence for sufficiently large $||\xi||_p$, let us say $||\xi||_p \geq n_0$, we obtain $$|\sigma (x,\xi) - \lambda| \geq \frac{C}{|\beta (x,\xi)|} \geq C',$$proving that $$\lambda \in \C \setminus \overline{\bigcup_{||\xi||_p \geq n_0} \{\sigma (x,\xi) \esp : \esp x \in \Z_p^d\}}.$$This concludes the proof.  
\end{proof}

For the case $T_\sigma \in Op(\Tilde{S}^m_{\rho,0} (\Z_p^d \times \widehat{\Z}_p^d))$, $m>0$, we can prove a similar result but we need the aditional condition of ellipticity.

\begin{coro}\label{Fredholmspectrumformulamgeq0}
Let $T_\sigma$ be a elliptic pseudo-differential operator with associated symbol in the H{\"o}rmander class $\Tilde{S}^m_{\rho,0} (\Z_p^d \times \widehat{\Z}_p^d)$, $m > 0$. Then, considering $T_\sigma$ as a bounded operator in $\mathcal{L} (H^{m}_2 (\Z_p^d), L^2 (\Z_p^d))$, we have$$Spec_F (T_\sigma) = \bigcap_{n \in \N_0} \overline{\bigcup_{||\xi||_p \geq n} \{\sigma (x,\xi) \esp : \esp x \in \Z_p^d\}}. $$
\end{coro}
\begin{proof}
Let us assume $$|\sigma (x , \xi)| \geq C \langle \xi \rangle^m, \esp \esp \text{for} \esp ||\xi||_p \geq n_0.$$ If we take $$\lambda \in \C \setminus \overline{\bigcup_{||\xi||_p \geq n_0} \{\sigma (x,\xi) \esp : \esp x \in \Z_p^d\}},$$then also $\sigma (x , \xi) - \lambda$ is also the symbol of an elliptic operator. Thus \cite[Theorem 8.6]{p-adicHormanderclasses} assure that $T_\sigma - \lambda \iota_{s+m , s}$ is a Fredholm operator. Conversely, if $$T^\bot T_\sigma - I , T_\sigma T^\bot - I \in Op(S^{-\infty} (\Z_p^d \times \widehat{\Z}_p^d)),$$ for some $T^\bot \in \mathcal{L}(H^{s}_2 (\Z_p^d) , H^{s+m}_2 (\Z_p^d))$ then  $$J_{m} T^\bot T_\sigma J_{-m} - I \in Op(S^{-\infty} (\Z_p^d \times \widehat{\Z}_p^d)),$$ and $T_\sigma J_{-m} \in Op(\Tilde{S}^0_{\rho,0} (\Z_p^d \times \widehat{\Z}_p^d))$, $J_{m} T^\bot \in \mathcal{L}(L^2 (\Z_p^d))$. Hence, with the same arguments as in Theorem \ref{Fredholmspectrumformula} we obtain $$|\sigma(x,\xi) - \lambda| \geq C \langle \xi \rangle^{m},$$for $||\xi||_p $ large enough, let us say $||\xi||_p \geq n_0 $, implying that $$\lambda \in \C \setminus \overline{\bigcup_{||\xi||_p \geq n_0} \{\sigma (x,\xi) \esp : \esp x \in \Z_p^d\}}.$$ 
\end{proof}
We can extend somehow the above corollary to $n$-hypoelliptic operators in the sense of \cite[Definition 8.3]{p-adicHormanderclasses} but for that it will be necessary to change our definition of Fredholm spectrum. The reason for doing that is that the inverse of a $n$-hypoelliptic operator modulo an smoothing operator does not have the necessary properties of boundedness between Sobolev spaces to apply a similar reasoning to the used in the proof of Corollary \ref{Fredholmspectrumformulamgeq0}. Thus we will have to give a new definition similar to the definition of Fredholm spectrum but in terms only of the algebra $Op(\Tilde{S}^\infty_{\rho , 0} (\Z_p^d \times \widehat{\Z}_p^d))$.
\begin{defi}
For a pseudo-differential operator $T_\sigma \in Op(\Tilde{S}^\infty_{0,0} (\Z_p^d \times \widehat{\Z}_p^d))$ define its associated pseudo-spectrum as $$Spec_\infty (T_\sigma):= \{\lambda \in \C \esp : \esp T_\sigma - \lambda I \esp \text{ is invertible in } \esp Op(\Tilde{S}^\infty_{0,0} (\Z_p^d \times \widehat{\Z}_p^d)) \esp \text{modulo} \esp Op(\Tilde{S}^{-\infty} (\Z_p^d \times \widehat{\Z}_p^d)) \}.$$
\end{defi}
With the above definition we obtain:
\begin{teo}\label{Fredholmspectrumformulanhypo}
Let $0 < \rho \leq 1$ be a real number. Let $T_\sigma$ be a $n$-hypoelliptic pseudo-differential operator with associated symbol in the H{\"o}rmander class $\Tilde{S}^m_{\rho,0} (\Z_p^d \times \widehat{\Z}_p^d)$. Then$$Spec_\infty (T_\sigma) = \bigcap_{n \in \N_0} \overline{\bigcup_{||\xi||_p \geq n} \{\sigma (x,\xi) \esp : \esp x \in \Z_p^d\}}. $$
\end{teo}

\section{\textbf{Hypoellipticity and Weyl law}}
In this section we use some ideas of the spectral theory in Agmon's book \cite{Agmonbook} to provide some information about the asymptotic behaviour for eigenvalues of hypoelliptic pseudo-differential operators. Specifically, the purpose of this section is to prove the following theorem, a version of the Weyl law for the asymptotic behaviour of eigenvalues of an $n$-hypoelliptic sectorial operator.

\begin{teo}\label{weyllaw}
Let $0 < \rho \leq 1$. Let $T_\sigma$ be a pseudo-differential operator with symbol $\sigma \in \Tilde{S}^m_{\rho,0} (\Z_p^d \times \widehat{\Z}_p^d)$ acting on $L^2 (\Z_p^d)$. For $N\in \N_0$ let us define $$A_N(\sigma):=  \overline{  \bigcup_{||\xi||\geq p^N}   \{ \sigma (x,\xi) \esp : \esp x \in \Z_p^d\}}.$$ Assume that $\sigma$ is $n$-hypoelliptic of order $0  < n \leq m,$ let us say $$|\sigma (x,\xi)| \gtrsim \langle \xi \rangle^{n}, \esp \esp || \xi ||_p \geq p^{N_0},$$ and that there exists a curve $\gamma: [0, \infty) \to \C$ with the following properties: \begin{enumerate}
    \item[(i)] $\lim_{s \to \infty } |\gamma(s)| = \infty.$ 
    \item[(ii)]For large $s \in [0 , \infty)$ it holds: $$\sup_{\xi \in \widehat{\Z}_p^d} \frac{||\widehat{\sigma}(\cdot , \xi)||_{\ell^1 (\widehat{\Z}_p^d)} - \big| \int_{\Z_p^d} \sigma (x , \xi) dx \big|}{\big| \int_{\Z_p^d} \sigma (x , \xi) dx  - \gamma (s) \big|} < 1,$$and $$\sup_{\xi \in \widehat{\Z}_p^d} \frac{||\widehat{\sigma}^* (\cdot , \xi)||_{\ell^1 (\widehat{\Z}_p^d)} - \big| \int_{\Z_p^d} \sigma^* (x , \xi) dx \big|}{\big| \int_{\Z_p^d} \sigma^* (x , \xi) dx  - \gamma (s) \big|} < 1.$$ 
    \item[(iii)] For large $s \in [0,\infty)$ and $N_0 \in \N_0$ as above, and some  $\alpha \in [0,1)$, it holds $$|\gamma(s)|^{1- \alpha} \leq C_\gamma dist(\gamma(s),A_{N_0}(\sigma)).$$
\end{enumerate} Then we have the estimate $$N(t) \lesssim t^{\frac{d + \alpha(4n - d) )}{n}} \leq t^{\frac{d}{n} + 4 \alpha},$$where $$N(t):= \sum_{|\lambda_k (T_\sigma)| \leq t}1 .$$In consequence $$\esp k^{\frac{n}{d+ \alpha(4n - d)}} =O(|\lambda_k(T_\sigma)|) .$$ 
\end{teo}
For the proof we follow \cite[Theorem 13.6]{Agmonbook}.
\begin{proof}
We divide the proof in several steps. We will assume that $n > d/4$. For the case when $n$ is small we can apply the same arguments for $T_\sigma^k$ where $k$ is a natural number such that $nk$ is large enough. Also we we will assume that for some $ n$ is large enough so the following holds: $$ ||f||_{L^\infty (\Z_p^d)} \lesssim ||f||_{H^{n}_2 (\Z_p^d)}^{2 - d/2n}.$$ 
\begin{enumerate}
    \item[(i)] We begin with a simple proposition:
    \begin{pro}\label{inverseequalspseudoinverseplus}
    Let $T_\sigma$ be a pseudo-differential operator with symbol $\sigma \in \Tilde{S}^m_{\rho,0} (\Z_p^d \times \widehat{\Z}_p^d)$. Assume that $T_\sigma$ is invertible. Then then $(T_\sigma)^{-1} = T_{1/\sigma} + R$, where $R$ is a smoothing operator.
    \end{pro}
    \begin{proof}
    If $(T_\sigma)^{-1}$ exists then $0 \in Res(T_\sigma) \subset Res_\infty (T_\sigma)$. By the formula in \ref{Fredholmspectrumformulanhypo} that means that the symbol $1/\sigma (x,\xi)$ is well defined for $||\xi||_p$ large enough so $$\esp \esp \esp \esp T_{1/\sigma} T_\sigma = I + R', \esp R' \in \Tilde{S}^{-\infty} (\Z_p^d \times \widehat{\Z}_p^d), \esp \esp \text{and} \esp \esp T_\sigma^{-1} T_\sigma = I \esp \implies (T_{1/\sigma} - T_\sigma^{-1}) T_\sigma = R', $$and thus $T_{1/\sigma} - T_\sigma^{-1} =R:= R' T_\sigma^{-1} \in Op(\Tilde{S}^{-\infty} (\Z_p^d \times \widehat{\Z}_p^d))$.
    \end{proof}
    \item[(ii)] Using Gershgorin Theorem, see \cite[Theorem 3.5]{Velasquez-Rodriguez2019}, we see that for large $s$ the operator $$(T_\sigma - \gamma (s) I)^{-1},$$ exists and is a bounded operator in $L^2(\Z_p^d)$. In particular $Res(T_\sigma)$ is non empty. Without loss of generality we suppose $0 \in Res(T_\sigma)$. Consider the linear operator $$T_s := (T_\sigma - \gamma (s) I) T_\sigma.$$By \cite[Lemma 8.9]{p-adicHormanderclasses} and Proposition \ref{inverseequalspseudoinverseplus} we get $T_s^{-1}= T_\sigma^{-1} (T_\sigma - \gamma (s) I)^{-1}  \in \Tilde{S}^{-2n}_{0,0} (\Z_p^d \times \widehat{\Z}_p^d)$. In this way Corollary \ref{corohilbertschmitpseudos} assures that $T_s^{-1}$ is a Hilbert-Schmit operator for every $s \in [0, \infty)$.
        \item[(iii)] Take any orthonormal basis $\{\psi_k\}_{k \in \N}$ for $L^2 (\Z_p^d)$. Then, for any sequence $\{ a_k \}_{k \in \N}$ of complex numbers and every fixed $x \in \Z_p^d$, using the Sobolev embedding we can estimate 
    \begin{align*}
    \Big| \sum_{k=1}^l a_k T_s^{-1}(\psi_k (x)) \Big| &= \Big| T_s^{-1} \Big( \sum_{k=1}^l a_k \psi_k (x) \Big) \Big| \leq || T_s^{-1} \Big( \sum_{k=1}^l a_k \psi_k (x) \Big) ||_{H^{n}_2 (\Z_p^d)}^{2 - d/2n} \\ & \lesssim ||T_s^{-1}||_{\mathcal{L}(L^2 (\Z_p^d), H^{n}_2 (\Z_p))}^{2 - d/n} \Big( \sum_{k=1}^l |a_k|^2 \Big)^{1/2}.
    \end{align*}In this way, if we take $a_k = \overline{A_s(\psi (x))}$, we obtain for any $x \in \Z_p^d$ $$\sum_{k=1}^l  |T_s^{-1}(\psi_k (x))|^2 \lesssim ||T_s^{-1}||_{\mathcal{L}(L^2 (\Z_p^d) , H^{n}_2 (\Z_p^d))}^{4 - d/n} . $$Integrating in both sides and taking the limit as $n \to \infty$ we get \begin{align*}
        \sum_{k \in \N} [s_k(T_s^{-1})]^2 &= \sum_{k \in \N}||T_s^{-1}(\psi_k (x))||_{L^2 (\Z_p^d)}^2 \\ &= ||T_s^{-1}||_{HS}^2 \\ & \lesssim ||T_s^{-1}||_{\mathcal{L}(L^2 (\Z_p^d) , H^{n}_2 (\Z_p^d))}^{4 - d/n}\\& \leq ||T_\sigma^{-1}||_{\mathcal{L}(L^2 (\Z_p^d) , H^{n}_2 (\Z_p^d))}^{4 - d/n} ||(T_\sigma - \gamma (s) I)^{-1}||_{\mathcal{L}(L^2 (\Z_p^d))}^{4-d/n} .
    \end{align*}Moreover, using Lemma \ref{inverseequalspseudoinverseplus} we obtain $$\sum_{k \in \N} [s_k(T_s^{-1})]^2 \lesssim ||T_{1/(\sigma - \gamma(s))}||_{\mathcal{L}(L^2 (\Z_p^d))}^{4 - d/n} $$ 
    \item[(iv)]The eigenvalues $\lambda_k(T_s^{-1})$ are exactly $\lambda_k(T_s^{-1}) =  \lambda_k(T_\sigma)^{-1}(\lambda_k (T_\sigma) - \gamma (s))^{-1}$. With this, applying the additive Weyl inequality \cite{weylinequality}, we deduce \begin{align*}
        \sum_{k \in \N} \frac{1}{|\lambda_k (T_\sigma)(\lambda_k(T_\sigma) - \gamma(s))|^2} &\leq \sum_{k \in \N} [s_k(T_s^{-1})]^2 \\ & \lesssim ||T_s^{-1}||_{\mathcal{L}(L^2 (\Z_p^d) , H^n_2 (\Z_p^d))}^{4 - d/n}\\ & \leq ||T_{1/(\sigma - \gamma(s))}||_{\mathcal{L}(L^2 (\Z_p^d))}^{4 - d/n}.
    \end{align*} 
    \item[(v)] In this point we need to estimate the operator norm of $T_{1/(\sigma - \gamma(s))}$ but this is easy because of Corollary \ref{coroequivopL2norm}: $$||T_{1/(\sigma - \gamma(s))}||_{\mathcal{L}(L^2 (\Z_p^d))} \lesssim \sup_{||\xi||_p \geq N_0} || 1/(\sigma(\cdot , \xi) - \gamma(s)) ||_{L^2 (\Z_p^d)} \leq |\gamma (s)|^{-(1-\alpha)}. $$With this we obtain $$\sum_{k \in \N} \frac{1}{|\lambda_k (T_\sigma)(\lambda_k(T_\sigma) - \gamma(s))|^2} \lesssim |\gamma (s)|^{-(1 - \alpha)(4 - d/n)},$$for large $s$.
    \item[(vi)] We use the same trick as in \cite[Theorem 13.6]{Agmonbook}: make $t= |\gamma(s)|$ and consider the eigenvalues $\lambda_k$ such that $|\lambda_k (T_\sigma)| \leq t$. Then, since $|\lambda_k (T_\sigma)(\lambda_k(T_\sigma) - \gamma(s))| \leq 2 t^2$, we get \begin{align*}
        \sum_{|\lambda_k (T_\sigma)| \leq t} \frac{1}{4 t^4} &\leq \sum_{|\lambda_k (T_\sigma)| \leq t} \frac{1}{|\lambda_k (T_\sigma)(\lambda_k(T_\sigma) - \gamma(s))|^2} \\ &\lesssim t^{-(1 - \alpha)(4 - d/n)}.
    \end{align*}Multiplying by $4t^4$ in both sides we obtain:$$N(t) = \sum_{|\lambda_k (T_\sigma)| \leq t} 1 \lesssim t^{\frac{d + \alpha(4n - d)}{n}}.$$
\end{enumerate}
The proof is complete.
\end{proof}
A particular class of $n$-hypoelliptic pseudo-differential operators satisfying the hypothesis of Theorem \ref{weyllaw} is given in the following definition:
\begin{defi}
We say that a pseudo-differential operator $T_\sigma \in Op(S^m_{\rho, 0} (\Z_p \times \widehat{\Z}_p))$ is sectorial if there exists $\theta_1 , \theta_2 \in [0,2\pi)$, $0<\theta_2 - \theta_1 < \pi/2$, such that $$\bigcup_{||\xi||_p \geq N} \{\sigma (x , \xi) \esp : \esp x \in \Z_p\} \subseteq \{z = r e^{i \theta} \in \C \esp : \esp r \geq 0 , \esp \esp \theta_1 \leq \theta \leq \theta_2 \},$$ for some $N \in \N_0.$
\end{defi}
For the above class of operators we can prove:

\begin{coro}
Let $0 < \rho \leq 1$. Let $T_\sigma$ be a pseudo-differential operator with symbol $\sigma \in \Tilde{S}^m_{\rho,0} (\Z_p^d \times \widehat{\Z}_p^d)$ acting on $L^2 (\Z_p^d)$. Assume that $T_\sigma $ is $n$-hypoelliptic, $0<n \leq m$, and that it is sectorial, let us say $$\bigcup_{|\xi|_p \geq N_0} \{\sigma (x , \xi) \esp : \esp x \in \Z_p\} \subseteq \{z = r e^{i \theta} \in \C \esp : \esp r \geq 0 , \esp \esp \theta_1 \leq \theta \leq \theta_2 \},$$ for some $N_0 \in \N_0$. Also assume that $$\langle \xi \rangle^m \lesssim \Big| \int_{\Z_p^d} \sigma (x , \xi) dx \Big|, $$for large $\xi$, and $$\bigcup_{||\xi||_p \geq N_1 } \Big\{ \int_{\Z_p^d} \sigma (x , \xi) dx \esp : \esp \xi \in \widehat{\Z}_p^d  \Big\} \subset \{z = r e^{i \theta} \in \C \esp : \esp r \geq 0 , \esp \esp \theta_1 \leq \theta \leq \theta_2 \}.$$ Then $$N(t) = \sum_{|\lambda_k (T_\sigma)| \leq t} 1 \lesssim t^{\frac{d }{n}}.$$
\end{coro}
\begin{proof}
Just take the curve $\gamma(s)= s e^{i ((\theta_2 + \theta_1)/2 + \pi)}$ and apply Theorem \ref{weyllaw}.
\end{proof}
\begin{coro}
Let $T_\sigma$ be a pseudo-differential operator with symbol $\sigma \in \Tilde{S}^m_{\rho,0} (\Z_p^d \times \widehat{\Z}_p^d)$ acting on $L^2 (\Z_p^d)$. Assume that $T_\sigma $ is $n$-hypoelliptic, $0<n \leq m$, and its associated symbol $\sigma (x , \xi)$ is a real valued function for every $\xi \in \widehat{\Z}_p^d$. Also assume that $$\langle \xi \rangle^m \lesssim \Big| \int_{\Z_p^d} \sigma (x , \xi) dx \Big|, $$for large $\xi$. Then $$N(t) = \sum_{|\lambda_k (T_\sigma)| \leq t} 1 \lesssim t^{\frac{d }{n}}.$$
\end{coro}
\begin{proof}
Just take the curve $\gamma(s)= s i $ and apply Theorem \ref{weyllaw}.
\end{proof}
\section{\textbf{Vilenkin groups}} 
In this final section our purpose is to extend the work of the previous sections and also the results in \cite{p-adicHormanderclasses} to a locally compact Vilenkin group $(G,+)$. We will begin by recalling the basics on Vilenkin groups.
\begin{defi}\normalfont
Throughout this section $G$ will denote a locally compact Vilenkin group,that is, a locally compact Abelian topological group containing a  strictly  decreasing  sequence  of  open  compact  subgroups $\{G_k\}_{k \in \Z}$ such that:
\begin{enumerate}
    \item[(i)] $\sup_{k \in \Z} \big| G_k /G_{k+1} \big|< 
    \infty.$
    \item[(ii)] $\bigcup_{k \in \Z} G_k =G , \esp \esp \text{and} \esp \esp \bigcap_{k \in \Z} G_k = \{0\}.$
\end{enumerate}
\end{defi}
Examples of such locally compact Vilenkin groups are the $p$-adic numbers and, more generally, the additive group of a local field. 

Let $\widehat{G}$ denote  the  dual  group of $G$ in the sense of Pontryagin. For the set $\widehat{G}$ of characteres $\chi_\xi$ on $G$ we will use the notation $\chi_{\xi_1} (x) \chi_{\xi_2} (x) = \chi_{\xi_1 + \xi_2} (x). $Also for functions $f: \widehat{G} \to \C$ we will write $f(\xi)$ instead of $f(\chi_\xi)$. Sometimes we will write $$\sum_{\xi \in \widehat{G}} f(\xi), \esp \esp \text{instead of}  \esp \sum_{\chi_\xi \in \widehat{G}} f(\chi_\xi).$$ For  each $k \in \Z$, let $\widehat{G}_k$ denote the annihilator  of $G_k$, that is, $$\widehat{G}_k := \{\chi_\xi \in \widehat{G} \esp : \esp \chi_\xi (x) = 1 \esp \text{on} \esp G_n \}.$$Then we have $\bigcup_{k \in \Z} \widehat{G}_k = G$, $\bigcap_{k \in \Z} \widehat{G}_k = \{1\},$ and $$|\widehat{G}_{k+1} / \widehat{G}_k|=|G_k/G_{k+1}|. $$We choose Haar measures $\mu $ on $G$ and $\nu$ on $\widehat{G}$ in such a way that $\mu(G_0) = \nu (\widehat{G}_0) =1$. Then we can see that $\mu(G_k) = \nu (\widehat{G}_k)^{-1}$ for all $k \in \Z$. 

Now let us define $i_k := \nu (\widehat{G}_k)$. Then we can construct a metric $d$ on $G$ as \[ d(x,y) := \begin{cases}
i_k^{-1} & \esp \text{if} \esp x - y \in G_k\setminus G_{k+1}, \\ 0& \esp \text{if} \esp x=y.
\end{cases}\]
Also we will use the notation $|x|_G := d(x , 0)$ and also $|x - y|_G := d(x,y)$. Then it is easy to see that $|x - y|_G \leq \max \{|x|_G , |y|_G \}. $ Similarly there is a metric $d'$n the dual group $\widehat{G}$ with the same properties given by \[ |\xi - \eta|_G := d(\xi,\eta) :=  \begin{cases}
i_k & \esp \text{if} \esp \xi - \eta \in \widehat{G}_{k+1}\setminus G_{k}, \\ 0& \esp \text{if} \esp \eta = \xi.
\end{cases}\]We set again $|\xi|_{\widehat{G}}:= d'(\xi , 0).$ Also we will use the notation $\langle \xi \rangle := \max \{ 1 ,|\xi|_{\widehat{G}}\}$. 
Examples of such locally compact Vilenkin groups are the $p$-adic numbers and, more generally, the additive group of a local field.

The definition of compact Vilenkin groups is subtlety different:

\begin{defi}\normalfont
A compact Vilenkin group $G$ is a compact Abelian topological group containing a  strictly  decreasing  sequence  of  open  compact  subgroups $\{G_k\}_{k \in \N_0}$, $G= G_0 \supset G_1 \supset... \supset \{0\}$, such that:
\begin{enumerate}
    \item[(i)] $\sup_{k \in \N_0} \big| G_n /G_{n+1} \big|< 
    \infty.$
    \item[(ii)] $ \bigcap_{k \in \N_0} G_k = \{0\}.$
\end{enumerate}
\end{defi}
The most important examples of compact Vilenkin groups are the balls centred at the origin of a local field, such as the compact group of $p$-adic integers.

For operator acting on function defined over a locally compact Vilenkin group $G$ there already exists a notion of symbol classes. See \cite{pseudosvinlekin, harmonicfractalanalysis}. 

\begin{defi}
Let $m \in \R$ and $0 \leq \delta \leq  \rho \leq 1$ be real numbers. A continuous function $\sigma : G \times  \widehat{G} \to \C$ belongs to the symbol class $\Check{S}^m_{\rho , \delta} (G \times  \widehat{G})$ if the following estimate holds: $$|\triangleplus_y^x \triangleplus_\eta^\xi \sigma (x , \xi)| \leq C_{m , \alpha , \beta, \rho , \delta } |y|_{G}^\beta |\eta|_{\widehat{G}}^\alpha \langle \xi \rangle^{m - \rho \alpha + \delta \beta},$$ for some constant $C_{ \rho , \delta } > 0$ and every $\alpha , \beta \in \N_0$, $|\eta|_{\widehat{G}} \leq \langle \xi \rangle$. 
\end{defi}

The above gives one a symbolic calculus: a composition formula and an adjoint formula. This calculus is very special and different from the usual calculus on the Euclidean case because the residue in the composition formula is an smoothing operator, in contrast with several pseudo-differential calculus in the Arquimidean setting where the residue is just a compact operator. The main goal of this section is to show how our definition of H{\"o}rmander classes has similar properties but also with other desired features.    
\subsection{\textbf{Locally compact Vilenkin groups}}
Let us begin by adjusting Definition \ref{defp-adicsobolev} to the non compact case: 
\begin{defi}\label{defsobolevvilenkin}\normalfont
 For our purposes the $L^r$-based \emph{Sobolev spaces} $H^s_r (G)$ are the metric completion of the collection of simple function on $G$ with the norm $$||f||_{H^s_r (G)} := ||D^s f||_{L^r (G)}.$$ Here we are using the notation $\langle \xi \rangle := \max\{1 , ||\xi||_p \},$ and $$D^s f (x):= \int_{\widehat{G}} \langle \xi \rangle^s \widehat{f}(\xi) \chi_\xi ( x) dx.$$We will call \emph{rapidly decreasing functions on} $G$ to the elements of the \emph{Schwartz space} $$ \mathcal{S} (G) := \bigcap_{s \in \R} H^s_2 (G).$$The class of \emph{tempered distributions on} $G$ is defined as $$\mathcal{S}'(G) := \bigcup_{s \in \R} H^s_2 (G).$$ 
 \end{defi}
We present now our definition of H{\"o}rmander classes on locally compact Vilenkin groups:
\begin{defi}\label{hormanderclasses}\normalfont
    \esp
  \begin{enumerate}
    \item[(i)] For functions $\varphi: \widehat{G} \to \C$ let us define the difference operator $\triangleplus$ as $$\triangleplus_\eta^\xi \varphi (\xi) := \varphi(\xi + \eta) - \varphi (\xi).$$
    \item[(ii)] Let $m \in \R$ and $0 \leq \delta \leq \rho \leq 1$ be given real numbers. We define the symbol classes $\Tilde{S}^m_{\rho, \delta} (G \times \widehat{G})$ as the collection of measurable functions $\sigma: G \times \widehat{G} \to \C$ such that, for all $\alpha , \beta \in \N_0$, the estimate $$|D^\beta_x \triangleplus_\eta^\xi \sigma (x, \xi )| \leq C_{\sigma, m, \alpha, \beta} |\eta|^\alpha_{\widehat{G}} \langle \xi \rangle^{m - \rho \alpha + \delta \beta } , $$ holds for some $C_{\sigma, m, \alpha, \beta,}>0$, and any $(x, \xi,\eta) \in G \times \widehat{G} \times \widehat{G},$ $|\eta|_{\widehat{G}} \leq \langle \xi \rangle.$ Given a symbol $\sigma (x , \xi)$ in a H{\"o}rmander class we define its associated pseudo-differential operator $T_\sigma$ by the formula $$T_\sigma f (x) := \int_{\widehat{G}} \sigma (x , \xi) \widehat{f} (\xi ) \chi_\xi (x) d \xi ,$$ where $\widehat{f}$ is the Fourier transform of the function $f$ in turn defined as $$\widehat{f} (\xi) := \int_{G} f (x) \overline{\chi_\xi (x)} dx.$$We will denote by $Op(\Tilde{S}^m_{\rho, \delta} (G \times \widehat{G}))$ the class of pseudo-differential operators with symbol in the H{\"o}rmander class $\tilde{S}^m_{\rho, \delta} (G \times \widehat{G})$. Furthermore, we define$$\Tilde{S}^{-\infty} (G \times \widehat{G}) := \bigcap_{m \in \R} \Tilde{S}^m_{0,0} (G \times \widehat{G}),$$ $$\Tilde{S}^\infty_{\rho,\delta} (G \times \widehat{G}) := \bigcup_{m \in \R} \Tilde{S}^m_{\rho, \delta} (G \times \widehat{G}).$$Sometimes we will denote by $Op(\Tilde{S}^\infty_{\rho, \delta} (G \times \widehat{G}))$ the class of pseudo-differential operators with symbol in the class $\Tilde{S}^\infty_{\rho, \delta} (G \times \widehat{G})$. 
  \end{enumerate}
 \end{defi}
For the above defined classes we have a symbolic calculus similar to \cite{p-adicHormanderclasses}: 
\begin{pro}
Let $0<\rho \leq 1$ be a given real number. Let $T_{\sigma_1}$, $T_{\sigma_2}$ be pseudo-differential operators with symbols in the Hörmander classes $\sigma_1 \in \Tilde{S}^{m_1}_{\rho , 0} (G \times \widehat{G})$,  $\sigma_2 \in \Tilde{S}^{m_2}_{\rho , 0} (G \times \widehat{G})$. Then :
    \begin{enumerate}
        \item[(i)] $T_\sigma T_\tau = T_{\sigma \tau} + R$, where $R \in Op({\Tilde{S}^{-\infty}} (G \times \widehat{G}))$ and we have $\sigma_1 \sigma_2 \in \Tilde{S}^{m_1 + m_2}_{\rho , 0} (G \times \widehat{G})$.
        \item[(ii)] If $\sigma \in \Tilde{S}^{m}_{\rho , 0} (G \times \widehat{G})$ then $T_\sigma^t , T_\sigma^* \in Op(\Tilde{S}^{m}_{\rho , 0} (G \times \widehat{G}))$.
    \end{enumerate}
\end{pro}
Also, using the multiplier theorem in \cite{pseudosvinlekin}, we can prove the Sobolev boundedness of operators in the $(1,0)$-class:
\begin{teo}
Let $1<r < \infty$ be a given real number. Let $T_\sigma$ be a pseudo-differential operator with symbol $\sigma \in \Tilde{S}^m_{1 , 0} ( G \times \widehat{G})$ and let $s \in \R$ be a given real number. Then $T_\sigma$ extends to a bounded operator from $H^{s+m}_r (\Z_p^d)$ to $H^{s}_r (\Z_p^d)$.
\end{teo}

Moreover, we can use the symbolic representation of operators acting on $L^2 (G)$ to obtain some spectral information:
\begin{teo}\label{Fredholmspectrumformulavilenkin}
Let $0 < \rho \leq 1$ be a real number. Let $T_\sigma$ be a pseudo-differential operator with associated symbol in the H{\"o}rmander class $\Tilde{S}^0_{\rho,0} (G \times \widehat{G})$. Then, considering $T_\sigma$ as a bounded operator on $L^2(G)$, we have$$Spec_F (T_\sigma) = \bigcap_{n \in \N_0} \overline{\bigcup_{|\xi|_{\widehat{G}} + |x|_G \geq n} \{\sigma (x,\xi) \esp : \esp x \in G\}}. $$
\end{teo}
However, in the compact case several technical issues doesn't exists and we can give a better spectral information.
\subsection{\textbf{Compact Vilenkin groups}}
The case of compact Vilenkin groups is completely analogous to the case of $\Z_p^d$. Thus we can extend easily all the results collected in this paper. For example we have:
\begin{teo}\label{uniform compactpseudosvilenkin}
Let $T_\sigma \in Op(S^0_{1 , 0} (G \times \widehat{G}))$ be a pseudo-differential operator. Them $T_\sigma$ extends to bounded operator in the ideal $\mathfrak{R}(L^r(G)) \cap Op(\Tilde{S}^0_{1,0} (G \times \widehat{G}))$ if and only if $$d_\sigma :=\limsup_{|\xi|_{\widehat{G}} \to \infty} || \sigma (\cdot , \xi)||_{L^\infty (G)} = 0.$$
\end{teo}

\begin{teo}
Let $G$ be a compact Vilenkin group. Let $T_\sigma \in Op(\Tilde{S}^0_{0,0}(G \times \widehat{G}))$ be a compact pseudo-differential operator. Assume that the sequence $\{||\sigma (x , \xi_k)||_{L^2 (G)}\}_{k \in \N_0}$ is ordered in a non increasing order. Then:
\esp 
\begin{enumerate}
    \item[(i)] $T_\sigma$ is a Dixmier traceable if and only if $$||T_\sigma||_{\mathcal{L}^{(1 , \infty)} (L^2 (G))}:= \sup_{N \geq 1} \frac{1}{\log(1 + N)} \sum_{0 \leq k \leq N} ||\sigma (\cdot , \xi_k)||_{L^2 (G)} < \infty.$$
    \item[(ii)] $T_\sigma$ is of $(s , w)$-type if and only if $$||T_\sigma||_{\mathfrak{L}_{s,w} (L^2 (G))} := \Big( \sum_{k \in \N_0} [k^{\frac{1}{r} - \frac{1}{w}} ||\sigma (\cdot , \xi_k)||_{L^2 (G)}]^w \Big)^{1/w}.$$
    \item[(iii)] $T_\sigma$ belong to the Schatten-von Newmann class $S_\gamma (L^2 (G))$ if and only if $$||T||_{S_\gamma (L^2 (G))} := \Big( \sum_{k \in \N_0}||\sigma (\cdot , \xi_k)||_{L^2 (G)}^\gamma \Big)^{1/\gamma}< \infty .$$
\end{enumerate}
\end{teo}
\section*{\textbf{Acknowledgments}}
The author thanks Professor Michael Ruzhansky for his help during the development of this work.
\nocite{*}
\bibliographystyle{acm}
\bibliography{main}

\begin{thebibliography}{10}

\bibitem{zunigagalindo2}
{\sc A.~Yu.~Khrennikov, S. V.~Kozyrev, W. A. Z.-G.}
\newblock {\em Ultrametric Pseudodifferential Equations and Applications}.
\newblock Encyclopedia of Mathematics and its Applications 168. Cambridge
  University Press, 2018.

\bibitem{Agmonbook}
{\sc Agmon, S.}
\newblock {\em Lectures on Elliptic Boundary Value Problems}.
\newblock Van Nost. Reinhold, 1965.

\bibitem{2002J}
{\sc {Avetisov}, V., {Bikulov}, A., {Kozyrev}, S., and {Osipov}, V.}
\newblock {p-adic models of ultrametric diffusion constrained by hierarchical
  energy landscapes}.
\newblock {\em Journal of Physics A Mathematical General 35\/} (Jan 2002),
  177--189.

\bibitem{2009JPhA...42h5003A}
{\sc {Avetisov}, V.~A., {Bikulov}, A.~K., and {Zubarev}, A.~P.}
\newblock {First passage time distribution and the number of returns for
  ultrametric random walks}.
\newblock {\em Journal of Physics A Mathematical General 42\/} (Feb 2009),
  085003.

\bibitem{Avetisov2014}
{\sc Avetisov, V.~A., Bikulov, A.~K., and Zubarev, A.~P.}
\newblock Ultrametric random walk and dynamics of protein molecules.
\newblock {\em Proceedings of the Steklov Institute of Mathematics 285}, 1 (Aug
  2014), 3--25.

\bibitem{fredholmalg}
{\sc B.A.~Barnes, e.~a.}
\newblock {\em Riesz and Fredholm Theory in Banach Algebras (Research Notes
  Inmathematics Series)}.
\newblock Pitman Publishing, 1982.

\bibitem{basisQp}
{\sc {Bikulov}, A.~K., and {Zubarev}, A.~P.}
\newblock {On one real basis for $L^2(Q\_p)$}.
\newblock {\em arXiv e-prints\/} (Apr. 2015).

\bibitem{butzer}
{\sc Butzer, P.~L., and Wagner, H.}
\newblock Walsh-fourier series and the concept of a derivative.
\newblock {\em Applicable Analysis 3}, 1 (1973), 29--46.

\bibitem{Butzer1975}
{\sc Butzer, P.~L., and Wagner, H.~J.}
\newblock On dyadic analysis based on the pointwise dyadic derivative.
\newblock {\em Analysis Mathematica 1}, 3 (Sep 1975), 171--196.

\bibitem{LpboundsDuvan}
{\sc {Cardona}, D.}
\newblock {On the boundedness of periodic pseudo-differential operators}.
\newblock {\em arXiv e-prints\/} (Jan 2017), arXiv:1701.08184.

\bibitem{Cardona2019}
{\sc Cardona, D., and Kumar, V.}
\newblock {$L^p$}-boundedness and {$L^p$}-nuclearity of multilinear
  pseudo-differential operators on {$\Z^n$} and the torus {$\T^n$}.
\newblock {\em Journal of Fourier Analysis and Applications\/} (Jul 2019).

\bibitem{LpboundsJulio}
{\sc Delgado, J.}
\newblock Lp-bounds for pseudo-differential operators on the torus.
\newblock {\em Pseudo-Differential Operators, Generalized Functions and
  Asymptotics\/} (2013), 103--116.

\bibitem{DELGADO2014779}
{\sc Delgado, J., and Ruzhansky, M.}
\newblock Kernel and symbol criteria for schatten classes and r-nuclearity on
  compact manifolds.
\newblock {\em Comptes Rendus Mathematique 352}, 10 (2014), 779 -- 784.

\bibitem{Delgado2018}
{\sc Delgado, J., and Ruzhansky, M.}
\newblock Fourier multipliers, symbols, and nuclearity on compact manifolds.
\newblock {\em Journal d'Analyse Math{\'e}matique 135}, 2 (Jun 2018), 757--800.

\bibitem{nuclearJulio}
{\sc Delgado, J., and Wong, M.~W.}
\newblock {$L^p$}-nuclear pseudo-differential operators on {$\Z$} and
  {$\mathbb{S}^1$}.
\newblock {\em Proceedings of the American Mathematical Society 141}, 11
  (2013), 3935--3942.

\bibitem{HAChina}
{\sc Der-Chen~Chang, Charles Fefferman~(auth.), M. C. D.-g. D. S. G. C.-C.
  Y.~e.}
\newblock {\em Harmonic Analysis in China}, 1~ed.
\newblock Mathematics and Its Applications 327. Springer Netherlands, 1995.

\bibitem{fisicapadica}
{\sc {Dragovich}, B., {Khrennikov}, A.~Y., {Kozyrev}, S.~V., {Volovich}, I.~V.,
  and {Zelenov}, E.~I.}
\newblock {$p$-Adic Mathematical Physics: The First 30 Years}.
\newblock {\em arXiv e-prints\/} (May 2017), arXiv:1705.04758.

\bibitem{gibbs1}
{\sc Gibbs, J., and Millard, M.}
\newblock {J.Walsh functions as solution of a logical differential equations}.
\newblock {\em NPL DES Rept. 1\/} (1969).

\bibitem{gibbs2}
{\sc Gibbs, J., and Millard, M.}
\newblock {J.Walsh functions as solution of a logical differential equations}.
\newblock {\em NPL DES Rept. 2\/} (1969).

\bibitem{sobolevmetrizablegroups}
{\sc Górka, P., a.~K.}
\newblock Sobolev spaces on metrizable groups.
\newblock {\em Annales Academiæ Scientiarum Fennicæ. Mathematica 40}, 2
  (2015), 837–849.

\bibitem{weylinequality}
{\sc Hinrichs, A.}
\newblock Optimal weyl inequality in banach spaces.
\newblock {\em Proceedings of the American Mathematical Society 134}, 3 (2006),
  731--735.

\bibitem{HandbookV1}
{\sc Johnson, B., and Lindenstrauss, J.}
\newblock {\em Handbook of the Geometry of Banach Spaces, Volume 1}, 1~ed.
\newblock North Holland, 2001.

\bibitem{traceclasscomm}
{\sc Kittaneh, F.}
\newblock Some trace class commutators of trace zero.
\newblock {\em Proceedings of the American Mathematical Society 113}, 3 (1991),
  655--661.

\bibitem{kochubeibook}
{\sc Kochubei, A.}
\newblock {\em Pseudo-Differential Equations and Stochastics Over
  Non-Archimedean Fields (Pure and Applied Mathematics)}, 1st~ed.
\newblock 2001.

\bibitem{p-adicball1}
{\sc Kochubei, A.~N.}
\newblock Heat equation in a p-adic ball.
\newblock {\em Methods Funct. Anal. Topology 2}, 3 (1996), 53--58.

\bibitem{p-adicball2}
{\sc Kochubei, A.~N.}
\newblock Linear and nonlinear heat equations on a p-adic ball.
\newblock {\em Ukrainian Mathematical Journal 70}, 2 (Jul 2018), 217--231.

\bibitem{littlwoodpaley}
{\sc Littlewood, J.~E., and Paley, R. E. A.~C.}
\newblock {Theorems on Fourier Series and Power Series (II)†}.
\newblock {\em Proceedings of the London Mathematical Society s2-42}, 1 (01
  1937), 52--89.

\bibitem{8585444}
{\sc Mantoiu, M., and Ruzhansky, M.}
\newblock Pseudo-differential operators, wigner transform and weyl systems on
  type i locally compact groups.
\newblock {\em DOCUMENTA MATHEMATICA 22\/} (2017), 1539--1592.

\bibitem{Marek}
{\sc Marek~Kowalski, Christopher~Sikorski, F.~S.}
\newblock {\em Selected topics in approximation and computation}, first edition
  and first printing~ed.
\newblock International Series of Monographs on Computer Science. Oxford
  University Press, 1995.

\bibitem{shahla}
{\sc Molahajloo, S.}
\newblock A characterization of compact pseudo-differential operators on
  $\mathbb{S}^1$.
\newblock {\em Pseudo-Differential Operators:Analysis, Applications and
  Computations 213\/} (2011), 25 -- 29.

\bibitem{Molahajloo2010}
{\sc Molahajloo, S., and Wong, M.~W.}
\newblock Ellipticity, fredholmness and spectral invariance of
  pseudo-differential operators on $\mathbb{S}^1$.
\newblock {\em Journal of Pseudo-Differential Operators and Applications 1}, 2
  (Jun 2010), 183--205.

\bibitem{Onneweer1977}
{\sc Onneweer, C.~W.}
\newblock Fractional differentiation on the group of integers of ap-adic
  orp-series field.
\newblock {\em Analysis Mathematica 3}, 2 (Jun 1977), 119--130.

\bibitem{Onneweer1978}
{\sc Onneweer, C.~W.}
\newblock {\em Differentiation on a p-Adic Or p-Series Field}.
\newblock Birkh{\"a}user Basel, Basel, 1978, pp.~187--198.

\bibitem{operatorsideals}
{\sc Pietsch, A.}
\newblock {\em Operator Ideals}.
\newblock North-Holland Mathematical Library 20. Elsevier, Academic Press,
  1980.

\bibitem{Pietsch}
{\sc Pietsch, A.}
\newblock {\em Eigenvalues and S-Numbers}.
\newblock Cambridge Studies in Advanced Mathematics 13. Cambridge University
  Press, 1987.

\bibitem{Pirhayati2011}
{\sc Pirhayati, M.}
\newblock Spectral theory of pseudo-differential operators on $\mathbb{S}^1$.
\newblock {\em Pseudo-Differential Operators: Analysis, Applications and
  Computations\/} (2011), 15--23.

\bibitem{LittPaleyQuek}
{\sc Quek, T.}
\newblock Littlewood–paley and multiplier theorems on weighted spaces over
  locally compact vilenkin groups.
\newblock {\em Journal of Mathematical Analysis and Applications 210}, 2
  (1997), 742 -- 754.

\bibitem{N-H.AnalysisRT1}
{\sc Ruzhansky, M., and Tokmagambetov, N.}
\newblock Nonharmonic analysis of boundary value problems.
\newblock {\em International Mathematics Research Notices 2016}, 12 (2016),
  3548--3615.

\bibitem{N-H.AnalysisRT2}
{\sc Ruzhansky, M., and Tokmagambetov, N.}
\newblock Nonharmonic analysis of boundary value problems without wz condition.
\newblock {\em Math. Model. Nat. Phenom. 12}, 1 (2017), 115--140.

\bibitem{ruzhansky1}
{\sc {Ruzhansky}, M., and {Turunen}, V.}
\newblock {\em Pseudo-Differential Operators and Symmetries: Background
  Analysis and Advanced Topics}, 1~ed.
\newblock Birkhäuser Basel, 2009.

\bibitem{pseudosvinlekin}
{\sc Saloff-Coste, L.}
\newblock {Operateurs pseudo-differentiels sur certains groupes totalement
  discontinus}.
\newblock {\em Studia Math\/} (1986), 205–228.

\bibitem{Dyadic}
{\sc Stankovic, R.~S., Butzer, P.~L., Schipp, F., Wade, W.~R., and Su, W.}
\newblock {\em Dyadic Walsh Analysis from 1924 Onwards Walsh-Gibbs-Butzer
  Dyadic Differentiation in Science Volumes 1 and 2}, 1~ed.
\newblock Atlantis Studies in Mathematics for Engineering and Science 12.
  Atlantis Press, 2015.

\bibitem{harmonicfractalanalysis}
{\sc Su, W.}
\newblock {\em Harmonic Analysis and Fractal Analysis Over Local Fields and
  Applications}.
\newblock World Scientific, 2017.

\bibitem{p-adicHormanderclasses}
{\sc Velasquez-Rodriguez, J.~P.}
\newblock {H{\"o}rmander classes of Pseudo-Differential operators over the
  compact group of {$p$}-adic integers}.
\newblock {\em arXiv e-prints\/} (Nov 2019), arXiv.

\bibitem{Velasquez-Rodriguez2019}
{\sc Velasquez-Rodriguez, J.~P.}
\newblock On some spectral properties of pseudo-differential operators on
  $\mathbb{T}$.
\newblock {\em Journal of Fourier Analysis and Applications\/} (Mar 2019).

\bibitem{wolf}
{\sc Wolf, F.}
\newblock On the essential spectrum of partial differential boundary problems.
\newblock {\em Communications on Pure and Applied Mathematics 12}, 2 (1959),
  211--228.

\bibitem{wongdfa}
{\sc {Wong}, M.~W.}
\newblock {\em Discrete Fourier analysis}, 1~ed.
\newblock Pseudo-Differential Operators 5. Birkhäuser Basel, 2011.

\bibitem{zunigagalindo1}
{\sc Zúñiga-Galindo, W.~A.}
\newblock {\em Pseudodifferential Equations Over Non-Archimedean Spaces}, 1st
  ed.~ed.
\newblock Lecture notes in mathematics 2174. Springer, 2016.

\end{thebibliography}
\Addresses

\end{document}